\documentclass[a4paper]{article}
\usepackage{fullpage}
\usepackage{yhmath}
\usepackage{setspace}
\usepackage{subfigure}
\usepackage{siunitx}
\usepackage{hyperref}
\usepackage{xcolor}

\usepackage{subcaption}

\usepackage{mathtools}
\usepackage{algorithmicx}
\usepackage{graphicx}
\usepackage{hyperref}
\usepackage{courier}
\usepackage{color}
\usepackage[english]{babel}
\usepackage{mathbbol}
\usepackage{multirow}

\usepackage{epsfig}
\usepackage{empheq}
\definecolor{Blue}{rgb}{0.3,0.3,0.9}
\usepackage{amsmath}
\usepackage{amsthm,amsbsy}
\usepackage{mathcomp}
\usepackage{textcomp}

\usepackage{amssymb}
\usepackage{graphicx}
\usepackage{dcolumn}
\usepackage{bm}
\usepackage{amsfonts}
\usepackage{latexsym}
\usepackage{pdfsync}
    \usepackage{fullpage} 
\usepackage{colordvi}
\usepackage{color}
\usepackage{booktabs}
\usepackage{graphicx}
\usepackage{stmaryrd}
\usepackage{authblk}
\usepackage{cite}
\usepackage{epstopdf}

\input xy
\xyoption{all}

\usepackage{algpseudocode}
\usepackage{graphics}
\usepackage{epsfig}

\usepackage[linesnumbered, ruled]{algorithm2e}
\SetKwRepeat{Do}{do}{while}%

\newtheorem{thm}{Theorem}[section]
\newtheorem{cor}[thm]{Corollary}
\newtheorem{prop}[thm]{Proposition}
\newtheorem{ex}[thm]{Example}
\newtheorem{rem}[thm]{Remark}

\newcommand{\commentout}[1]{}

\newcommand{\nwc}{\newcommand}

\newcommand{\kA}{\widetilde{A}}

\nwc{\bR}{\mb R}
\nwc{\bH}{{\mb H}}
\nwc{\bxp}{{{\mathbf x}}}
\nwc{\bap}{{{\mathbf y}}}

\nwc{\bPhi}{\mathbf{\Phi}}
\nwc{\bPsi}{\mathbf{\Psi}}
\nwc{\bh}{\mathbf h}

\nwc{\bI}{\mathbf I}
\nwc{\bP}{\mathbf P}
\nwc{\bs}{\mathbf s}
\nwc{\bd}{\mathbf{d}}
\nwc{\bX}{\mathbf X}

\nwc{\om}{\beta'}

\nwc{\nwt}{\newtheorem}
\nwc{\xp}{{x^{\perp}}}
\nwc{\yp}{{y^{\perp}}}
\nwt{remark}{Remark}
\nwt{definition}{Definition}
\nwt{corollary}{Corollary} 

\nwc{\ba}{{\mb a}}
\nwc{\bal}{\begin{align}}
\nwc{\ben}{\begin{equation*}}
\nwc{\beqq}{\begin{equation}}
\nwc{\bea}{\begin{eqnarray}}
\nwc{\beq}{\begin{eqnarray}}
\nwc{\bean}{\begin{eqnarray*}}
\nwc{\beqn}{\begin{eqnarray*}}
\nwc{\beqast}{\begin{eqnarray*}}

\nwc{\eal}{\end{align}}
\nwc{\een}{\end{equation*}}
\nwc{\eeqq}{\end{equation}}
\nwc{\eea}{\end{eqnarray}}
\nwc{\eeq}{\end{eqnarray}}
\nwc{\eean}{\end{eqnarray*}}
\nwc{\eeqn}{\end{eqnarray*}}
\nwc{\eeqast}{\end{eqnarray*}}

\nwc{\vep}{\varepsilon}
\nwc{\ep}{\epsilon}
\nwc{\ept}{\epsilon}
\nwc{\vrho}{\varrho}
\nwc{\orho}{\bar\varrho}
\nwc{\ou}{\bar u}
\nwc{\vpsi}{\varpsi}
\nwc{\lamb}{\lambda}
\nwc{\Var}{{\rm Var}}

\nwt{proposition}{Proposition}
\nwt{theorem}{Theorem}
\nwt{summary}{Summary}
\nwt{lemma}{Lemma}
\nwc{\nn}{\nonumber}
\nwc{\mf}{\mathbf}
\nwc{\mb}{\mathbf}
\nwc{\ml}{\mathcal}

\nwc{\IA}{\mathbb{A}} 
\nwc{\bi}{\mathbf i}
\nwc{\bo}{\mathbf o}
\nwc{\IB}{\mathbb{B}}
\nwc{\IC}{\mathbb{C}} 
\nwc{\ID}{\mathbb{D}} 
\nwc{\IM}{\mathbb{M}} 
\nwc{\IP}{\mathbb{P}} 
\nwc{\II}{\mathbb{I}} 
\nwc{\IE}{\mathbb{E}} 
\nwc{\IF}{\mathbb{F}} 
\nwc{\IG}{\mathbb{G}} 
\nwc{\IN}{\mathbb{N}} 
\nwc{\IQ}{\mathbb{Q}} 
\nwc{\IR}{\mathbb{R}} 
\nwc{\IT}{\mathbb{T}} 
\nwc{\IZ}{\mathbb{Z}} 

\nwc{\cE}{{\ml E}}
\nwc{\cP}{{\ml P}}
\nwc{\cQ}{{\ml Q}}
\nwc{\cL}{{\ml L}}
\nwc{\cX}{{\ml X}}
\nwc{\cW}{{\ml W}}
\nwc{\cZ}{{\ml Z}}
\nwc{\cR}{{\ml R}}
\nwc{ \cV}{{\ml V}}
\nwc{\cT}{{\ml T}}
\nwc{\crV}{{\ml L}_{(\delta,\rho)}}
\nwc{\cC}{{\ml C}}
\nwc{\cO}{{\ml O}}
\nwc{\cA}{{\ml A}}
\nwc{\cK}{{\ml K}}
\nwc{\cB}{{\ml B}}
\nwc{\cD}{{\ml D}}
\nwc{\cF}{{\ml F}}
\nwc{\cS}{{\ml S}}
\nwc{\cM}{{\ml M}}
\nwc{\cN}{{\ml N}}
\nwc{\cG}{{\ml G}}
\nwc{\cH}{{\ml H}}
\nwc{\bk}{{\mb k}}
\nwc{\bn}{{\mb n}}
\nwc{\cbz}{\overline{\cB}_z}
\nwc{\supp}{{\hbox{supp}}}
\nwc{\fR}{\Re}
\nwc{\bY}{\mathbf Y}

\nwc{\pft}{\cF^{-1}_2}
\nwc{\bU}{{\mb U}}
\nwc{\bG}{{\mb G}}

\nwc{\gd}{\textrm{grad }}
\nwc{\Hs}{\textrm{Hess  }}
\nwc{\sy}{\textrm{sym  }}
\nwc{\Pj}{\textrm{Proj  }}

\nwc{\bg}{\mathbf{g}}
\nwc{\mbf}{\mathbf{f}}
\nwc{\mbg}{\mathbf{g}}
\nwc{\mbh}{\mathbf{h}}
\nwc{\mbm}{\mathbf{m}}
\nwc{\mbk}{\mathbf{k}}
\nwc{\mbs}{\mathbf{s}}

\nwc{\mbe}{\mathbf{e}}
\nwc{\be}{\mathbf{e}}
\nwc{\Om}{\beta'}
\nwc{\ind}{\operatorname{I}}
\nwc{\mbx}{\mathbf{f}}
\nwc{\bb}{\mathbf{g}}
\nwc{\xmax}{f_{\rm max}}
\nwc{\xmin}{f_{\rm min}}
\nwc{\suppx}{\hbox{\rm supp} (\mbf)}
\nwc{\by}{\mathbf{h}}
\nwc{\bZ}{\mathbf{Z}}
\nwc{\bF}{\mathbf{F}}
\nwc{\bE}{\mathbf{E}}
\nwc{\bV}{\mathbf{V}}
\nwc{\cI}{\IZ^2_N}
\nwc{\chis}{{\chi^{\rm s}}}
\nwc{\chii}{{\chi^{\rm i}}}
\nwc{\pdfi}{{f^{\rm i}}}
\nwc{\pdfs}{{f^{\rm s}}}
\nwc{\pdfii}{{f_1^{\rm i}}}
\nwc{\pdfsi}{{f_1^{\rm s}}}
\nwc{\thetatil}{{\tilde\theta}}
\nwc{\red}{\color{red}}
\nwc{\prox}{\hbox{prox}} 
\nwc{\diag}{\hbox{\rm diag}}
\nwc{\sloc}{J_{\rm f}}
\nwc{\bu}{\xi}
\nwc{\bv}{\beta}
\nwc{\cU}{\mathcal{U}}
\nwc{\bN}{\mathbf{N}}
\nwc{\bw}{\mathbf{w}}
\nwc{\im}{i}
\nwc{\bom}{\mathbf{w}}
\nwc{\bt}{\mathbf{t}}
\nwc{\z}{y}
\nwc{\cY}{\mathcal{Y}}
\nwc{\wLam}{{\widetilde \Lambda}}

\title{Sequential subspace methods on  Stiefel manifold optimization \thanks{
Funding: The research  of the  author is supported  by grant 113-2115-M-007-015-
MY3 and 114-2112-M-029-008 from the Ministry of Science and Technology, Taiwan.
}
}

\author[1]{ Pengwen Chen\thanks{
 pengwen@math.nthu.edu.tw.
}
}
\author[2]{Chung-Kuan Cheng}
\author[3]{Chester Holtz}
\affil[1]{Department of Mathematics, National Tsing Hua  University, Hsinchu,  Taiwan}
\affil[2]{CSE and $^\ddag$ECE Departments, UC San Diego, La Jolla, CA, USA }
\affil[3]{CSE Department, UC San Diego, La Jolla, CA, USA }

\begin{document}

\maketitle
\begin{abstract} 

We investigate  the minimization of a quadratic function over Stiefel manifolds (the set of all orthogonal $r$-frames in $\IR^n$), which has applications in high-dimensional semi-supervised classification tasks. To reduce the computational complexity, we employ sequential subspace methods (SSM) to reduce the high-dimensional problem to a sequence of low-dimensional subproblems. In this paper, our goal is to achieve an optimal solution of high quality, referred to as  a “qualified critical point". Qualified critical points are defined as those where  the associated multiplier matrix meets specific upper-bound conditions. These points exhibit near-global optimality in quadratic optimization problems. 

In the context of a general quadratic, SSM generates a sequence of “qualified critical points” through low-dimensional “surrogate regularized models”. The convergence to a qualified critical point is guaranteed when each SSM subspace is constructed from the following vectors: (i) a set of orthogonal unit vectors associated with the current iterate, (ii) a set of vectors representing the gradient of the objective, and (iii) a set of eigenvectors associated with the smallest $r$ eigenvalues of the system matrix. Furthermore, incorporating Newton direction vectors into the subspaces can significantly accelerate the convergence of SSM.
\end{abstract}


\section{Introduction}

Optimization problems on smooth manifolds are common in science and engineering
due to  their natural geometric properties,  and they have a range of applications in fields such as machine learning, computer vision, robotics, scientific computing,  and signal processing~\cite{edelman_geometry_1998,  absil_projection-like_2012, absil_optimization_2008, Michor2007, bernard2021sparse, Huang2025}. 
 Let $I_r\in \IR^{r\times r}$ denote the identity matrix of dimension $r$.  For $n\ge r$, let $St(n,r)$ denote the 
 Stiefel manifold, defined as \beqq
St(n,r)
:=\{X\in \IR^{n\times r}: X^\top X=I_r\}.
\eeqq
In brief,  $St(n,r)$ is the set of matrices in $\IR^{n\times r}$ whose columns are orthonormal in $\IR^n$ with respect to 
the inner product $ \langle x,y\rangle=tr(x^\top y)$.
In this paper, we employ a subspace method to  solve  the problem
\begin{equation}\label{main_P}
    \min_{X \in St(n, r)} \{\mbf(X) = \frac{1}{2}tr(X^\top A X C) - tr(B^\top X)\},
\end{equation}
where $A$ is symmetric, $C$ is symmetric  positive definite, and $B\in \IR^{n\times r}$. 
For a symmetric matrix $A\in \IR^{n\times n}$,
perform    an eigendecomposition of $A$,
 \beqq
A=[v_1, v_2, \ldots, v_n]\diag(d_1, \ldots, d_n)[v_1, v_2, \ldots, v_n]^\top.
\eeqq
Assume the spectral gap $d_{r+1}>d_r$ in $A$.
Let $V_g$ denote the matrix whose columns are the ground eigenvectors $\{v_1,\ldots, v_r\}$.
When $C=I_r$, the problem in  (\ref{main_P}) can be simplified to 
the (unbalanced) Procrustes problem\cite{Schnemann1966AGS,Elden1999,Zhang2007}. 
This task can also  be regarded as a relaxation of one NP-hard quadratic assignment problem (Koopmans-Beckman problem) \cite{QAP}, where a feasible  discrete solution can be characterized as a point on the Stiefel manifold. 
Our primary aim is to develop an algorithm that computes a global minimizer of (\ref{main_P}). 
The nonconvexity introduced by the orthogonality constraints significantly complicates the task.

 Although   quadratic functions are relatively  simple, large-scale quadratic optimization with $n\gg r$, coupled with orthogonality constraints, represents a crucial category of matrix optimization challenges that have emerged  in various fields, including machine learning, statistics, and signal processing.
  For instance, the problem~(\ref{main_P})
with $C=I_r$ and $B=0$ is the core task in 
 Principal Component Analysis (PCA), a standard tool for dimensionality reduction and visualization~\cite{doi:10.1098/rsta.2015.0202, pmlr-v37-shamir15}.
   Specific instances of \eqref{main_P} also appear in  a widely used technique in signal processing,
     independent component analysis \cite{Comon1994,Nishimcri,Ablin2017, Ablin2023},
          where the search for a demixing (uncorrelated) matrix is reduced to the separation of orthogonal signals under proper statistical principles. 
  More recently,  in the realm of deep learning, the weights of a layer are often parametrized by an orthogonal matrix to overcome the difficulty in training deep networks\cite{Arjovsky2016,Bansal2018}.

The matrix $A$ mentioned in (\ref{main_P}) is derived from the graph Laplacian associated with a data set. The 
 eigenvectors, referred to as  ``ground eigenvectors'', corresponding to  the smallest eigenvalues,   reveal 
the inherent geometric properties of the graph, e.g., \cite{ng2001spectral,Belkin2003}.
The  Laplacian Eigenmap framework~\cite{Belkin2003}  computes the spectrum of a matrix derived from an underlying data graph to identify various characterizations of the data,  including   its cluster structure and geometric characteristics. 
A large body of subsequent work,  including \cite{Coifman2005}, uses the eigenvectors of the graph Laplacian for dimensionality reduction and data representation, in both unsupervised and semi-supervised contexts \cite{Tenenbaum2000, Zhou2003, Belkin2003, Belkin2004, Belkin2006, NIPS_Chester}. 
This paper focuses on a prominent application of the task in \eqref{main_P} within the context of graph machine learning, namely, semi-supervised graph embedding. 
This approach is significant because 
collecting labeled data can be costly and time-consuming, while   a wealth of unlabeled data is often readily  available. 
In semi-supervised graph embedding, 
our goal is to gain insights into the clustering structure by leveraging 
 additional label information, even when    only a small amount of pre-specified labeled data is provided.  
  The minimization problem in (\ref{main_P}) serves as a basic  model for  $r$-way classification with partial labeling. (See  section~\ref{sec4.2} for the details.) Unlike the  eigenvectors discussed in~\cite{Belkin2003}, the minimizer $X$ in the context is interpreted as a  locally biased perturbation of the graph Laplacian eigenspace, which is influenced by the  perturbation term $tr( B^\top X)$ in  (\ref{main_P}).

 When the problem dimension $n$ is moderate, various  algorithms already available  in the literature
 can be used to solve 
 the constrained optimization problem in (\ref{main_P}). 
  For instance,  the 
 gradient projection method is an effective method, especially suitable for  projections on simple constrained sets~\cite{Goldstein1964, Levitin1966, Bertsekas1976, bertsekas_nonlinear_2016}.   
From the perspective of manifold optimization, 
 the Stiefel manifold in (\ref{main_P})  is a smooth manifold with special quadratic constraints. In this context, 
 one can implement conjugate gradient methods or Newton methods over geodesic paths on the manifold\cite{edelman_geometry_1998}. 
The computational complexity of geodesics can be reduced by using a  line-search procedure on tangent spaces with proper retractions\cite{Manton2002,absil_optimization_2008
} or by using the Cayley transform to construct a feasible curve on the manifold\cite{Wen2013}. When applying  
 the inverse Hessian is costly, a quasi-Newton method for  Riemannian optimization
 can be utilized,  as demonstrated in \cite{doi:10.1137/18M121112X, doi:10.1137/140955483}.  Riemannian optimization methods have been implemented in
the software package Manopt\cite{manopt
}.

 This paper focuses on applying subspace methods to address the large-scale problems in (\ref{main_P}) with  $n\gg r$.   
  For $r=1$,
the  task involves minimizing 
 a quadratic function  over a unit sphere--a challenge encountered in trust region methods~\cite{sorensen_newtons_1982}\cite{conn_trust_2000}.  The
authors of \cite{hager_global_2005} introduce  Sequential Subspace Methods (SSM) as a means  to compute a global minimizer for  large-scale problems.
 {  Subspace techniques have been widely employed in numerical linear algebra, as they enable the next iteration to be generated within a low-dimensional subspace~\cite{doi:10.1142/9789812709356_0012}. } This approach reduces the computational cost of each iteration by solving associated low-dimensional subproblems, offering an effective solution to large-scale optimization challenges
 \cite{CSIAM-AM-2-4}.
  We extend the application of Sequential Subspace Methods (SSM) to optimization problems on the Stiefel manifold where $ r > 1 $. When the problem dimension $ n $ is large, it is common to encounter multiple local solutions, most of which are saddle points that are distant from the global minimizers. To ensure solution quality, we solve a sequence of regularized problems, and their global minimizers can be easily identified by computing the qualified critical points. 
 
While our discussion primarily focuses on quadratic functions for clustering applications, the results presented in this paper can also offer insights into minimizing general objectives over Stiefel manifolds. Below, we outline the key contributions of this paper.

\begin{itemize}
\item 
We analyze the optimality conditions  in~(\ref{main_P}) by examining the optimization over  the convex relaxation of $St(n,r)$. To our knowledge, there is no established systematic method for finding  the global minimizer  in~(\ref{main_P}), although some sufficient conditions where $ C = I_r $  have been reported in [EP99, ZQD07]. As  candidates for global minimizers, we introduce the concept of { ``}qualified critical points". These  points satisfy the first-order optimal condition, and the associated multiplier matrix $ \Lambda $ fulfills  $ \Lambda \preceq d_r C $. For a system matrix $A$ with identical ground eigenvalues, any qualified critical point  in~(\ref{main_P}) is guaranteed to be  a global minimizer.


\item 
SSM constructs a sequence of subspaces and solves the corresponding low-dimensional subproblems.
When the subspaces   are  spanned by columns of $X$, $AXC-B$, and ground eigenvectors of $A$,
 subproblem solutions converge to a qualified critical point of (\ref{main_P}). 
To achieve fast convergence, we can incorporate a Newton direction within the subspace. However, calculating Newton steps is often the most time-consuming aspect of this process. To address this issue, we impose orthogonality against the ground eigenvectors during the computation of the Newton step. This method improves the conditioning of the Hessian operator and reduces the total number of conjugate gradient (CG) iterations needed.
 
\item
In applications, there is a need to address a “locally biased” problem, such as finding a partition or clustering that is guided by a predefined "ground truth" partition of a seed set of vertices. These locally biased problems can be particularly challenging for popular eigenvector-based machine learning and data analysis tools. 
We demonstrate the effectiveness of the proposed SSM  on various artificial test cases. The quality of the local solutions obtained for these problems directly impacts classification accuracy. Numerical studies indicate that the SSM achieves high-quality clusters, both in terms of balanced graph cut metrics and the accuracy of labeling assignments across several real-world datasets. These improvements are consistent with SSM attaining lower objective values and smaller first-order residuals.
%

 \end{itemize}

 The paper is organized as follows.  In section \ref{sec:2},  we list some known results in the spherical case, i.e., $St(n,1)$. Next, we introduce qualified critical points and study one convex relaxation of (\ref{main_P}), which motivates the definition of regularized problems.   In section~\ref{sec:3}, we describe each ingredient of subspaces in SSM and show the convergence property. Finally, we demonstrate numerical results, which validate the effectiveness of the proposed algorithms in solving one $r$-way classification problem, described in section~\ref{sec4.2}.

Throughout this paper,  we denote $r$ as a positive integer that is significantly less  than $n$.
We use $\IR^{n}$ to represent  the $n$-dimensional real vector space. The inner product 
$\langle A,B\rangle$ is defined as the trace of the matrix $AB^\top$ for $A,B$ of the same size. 
Let  $\|X\|$ denote the Frobenius norm for a matrix $X$. Let $I$ denote the identity matrix, and $I_n$ denote the identity matrix with size $n\times n$.
Consider the Riemannian metric, inner product on $St(n,r)$,
\beqq
\langle u,v\rangle=tr(u^\top v),\;\textrm{ for any } u,v\in St(n,r).
\eeqq 
Let $\Lambda_{sym}$ denote the symmetric matrix $(\Lambda+\Lambda^\top )/2$ for a square matrix $\Lambda$.  Let $\diag([x_1,\ldots, x_n])$ denote the diagonal matrix with $\{x_1,\ldots, x_n\}$ lying on the main diagonal. 
Let $\cO_r$ be the orthogonal group, i.e., $Q\in \cO_r$ if and only if $Q\in \IR^{r\times r}$ and $Q^\top Q=I_r$. Let $e_i$  denote the canonical basis vector $[0,\ldots, 1,0,\ldots]$, whose entries are all $0$, except for the $i$-th entry being $1$. Let $1_n$ denote the $n$-dimensional column vector with entries all ones.

%
%

\section{Local optimality conditions}\label{sec:2}
%
 
\subsection{Preliminaries: Quadratic minimization over a sphere} 
Consider the application of the trust region method to the minimization \beqq
\min_z f(z), \; z\in \IR^n.\eeqq
Let $\Delta>0$ be the trust-region radius. 
 The trust region method updates the current iterate $z=z_k$ to $z_{k+1}=z_k+x$, where the update vector $x\in \IR^n$ solves  a quadratic minimization subproblem~\cite{sorensen_newtons_1982}\cite{conn_trust_2000}\cite{hager_global_2005}\cite{nocedal_numerical_2006},
\beqq\label{eq_1d'}
\min_x \{ \frac{x^\top A x}{2}-\langle x, b\rangle: \|x\|\le \Delta, x\in \IR^n\},
\eeqq
 the vector $b\in \IR^n$ is the differential $\partial f$ at the current iterate $z=z_k\in\IR^n$,  and   the symmetric  matrix $A$ is the Hessian matrix $\partial^2 f$. Let $\lambda\le 0$ be the Lagrangian multiplier associated with the constraint $\|x\|^2-\Delta^2\le 0$. 
{ The first order necessary condition for a critical point $x$  is 
 \beqq
( A-\lambda I) x=b\; \textrm{ with } \;  \lambda(\|x\|^2-\Delta^2)=0.
 \eeqq
For  a minimizer $x$, 
  the associated multiplier $\lambda$ must be bounded above  by   the smallest eigenvalue $d_1$  of $A$(Lemma 2.4\cite{sorensen_newtons_1982}). 
 To proceed, we concentrate on the sphere case,
\beqq\label{eq_1s'}
\min_x \{ \frac{x^\top A x}{2}-\langle x, b\rangle: \|x\|=1, x\in \IR^n\}.
\eeqq
where we take  unit radius $\Delta=1$ for simplicity.
The following  propositions indicate that  the global solution in (\ref{eq_1s'}) 
is the critical point associated  with   $\lambda\le d_1$.\footnote{The proof was given in  Lemma 2.8(ii) \cite{sorensen_newtons_1982}.}
Note that the problem in (\ref{eq_1s'})  is exactly the special case of (\ref{main_P}) with $r=1$.}

\begin{prop}{( Lemma 2,1 \cite{hager_minimizing_2001})}\label{1.3}
A vector $x\in \IR^n$ is a global solution of (\ref{eq_1s'}), if and only if $\|x\|=1$
and \beqq \label{eq23'}
A-\lambda I\succeq 0, (A-\lambda I)x=b
\eeqq 
holds for some $\lambda\in \IR$. 
\end{prop}
 
\begin{prop}( { Lemma 2.2, \cite{hager_minimizing_2001}})\label{1.4}  Consider the eigenvector decomposition \beqq
A=V_A \diag([d_1, d_2, \ldots, d_n]) V_A^\top,\; V_A=[v_1, \ldots, v_n].\eeqq
Let $V_g$ be the matrix whose columns are  the ``ground eigenvectors", i.e.,   $\{v_j: d_j=d_1\}$ of $A$. Then a solution $x$ can be determined in the following way.
\begin{itemize}
\item Degenerate case: Suppose \beqq\label{1dc}
V_g^\top b=0 \textrm{ and } 
c:=\|(A-d_1 I)^{\dagger} b\|\le 1.\eeqq
 Then $\lambda=d_1$ and \beqq 
x=(1-c^2)^{1/2} v+(A-d_1 I)^{\dagger} b
\eeqq
for any unit eigenvector $v$ of $A$ associated with eigenvalue $d_1$.
\item Nondegenerate case: If the condition in (\ref{1dc}) does not hold, then  a solution is $x=(A-\lambda I)^{-1} b$ for 
some $\lambda<d_1$
with $
\|x\|=1$. 

\end{itemize}
\end{prop}
\begin{rem} 
Note that 
 for $\lambda< d_1$, 
 $\|(A-\lambda I)^{-1} b\|$ decreases monotonically  with respect to $\lambda$. Since  $\lambda$ in (\ref{eq23'}) is  determined  to meet that condition $\|x\|=1$, we can obtain  a tighter bound on $\lambda$ from 
\beqq
(d_1-\lambda)^{-2}\|V_g b\|^2  \le 1=\|(A-\lambda I)^{-1} b\|^2\le (d_1-\lambda)^{-2}\|b\|^2.
\eeqq
That is, $\lambda$ lies in the interval  $ [ d_1-\|b\|, d_1-\|V_g b\|]$.
 \end{rem}

\subsection{Optimality conditions}
 
 Return to  the problem in (\ref{main_P}).
To proceed, we adopt standard approaches  in manifold optimization\cite{Manton2002, absil_optimization_2008, boumal_introduction_2022}.
Let  $\cM:=St(n,r)$ and
introduce two projections $\cP_{\cM}$ and $\Pj$.
 { For each $Y\in \IR^{n\times r}$ with rank $r$, 
 let  the
 reduced SVD  be $Y=U_Y D_Y V_Y^\top$. The
  polar projection  $\cP_{\cM}$
 on 
the manifold $\cM$ is given by
  \beqq \cP_\cM (Y)=U_Y V_Y^\top.\eeqq
  Equivalently,
the projection can be expressed as
\beqq\label{eq_14}
\cP_{\cM} (Y)=Y(Y^\top Y)^{-1/2}.
\eeqq}
For each $X\in St(n,r)$,
let  $X_\bot\in \IR^{n\times (n-r)}$ complete the orthonormal basis in $\IR^n$.
{ Let $Sym(r),Skew(r)$ denote  symmetric and skew-symmetric real matrices of size $r\times r$, respectively.}
Let $T_X St(n,r), N_X St(n,r)$ denote
the tangent space and the normal space  at $X$, respectively, 
\beqq\label{T_X}
T_X St(n,r)=\left\{
X\Delta_0+X_\bot \Delta_1: \; \Delta_0\in Skew(r),\; \Delta_1\in \IR^{(n-r)\times r}
\right\},
N_X St(n,r)=\left\{
X\Delta_2: \; \Delta_2\in Sym(r)
\right\}.
\eeqq

   \begin{definition}
     For each $U\in \IR^{n\times r}$, define the projection $\Pj_X$ on the tangent space $T_X St(n,r)$,
     \beqq
     \Pj_X U=U-X\, (X^\top U)_{sym},\eeqq
     where  $\Lambda_{sym}$ denotes $(\Lambda+\Lambda^\top)/2$ for each $\Lambda\in\IR^{r\times r}$.

     \end{definition}

For  a point $X\in \cM$  and a vector $v\in T_X\cM$, we
introduce 
the metric projection retraction 
\beqq
\cR_X(v)=arg\min_{y\in \cM} \|X+v-y\|^2=\cP_{\cM}(X+v),
\eeqq
and construct  a retraction curve $c(t)=\cR_X(tv)$.  In the following, if no confusion, we shall write $\mbf(X)$ for $\mbf(X; A,B,C)$.
 For each tangent $v$ at $T_{X} \cM$, $\cR_X(v)$ is a second-order retraction { (See Prop. 5.54 and 5.55  in~\cite{boumal_introduction_2022})}.
 Let $\gd \mbf$  denote
  the Riemannian gradient 
  and $\Hs \mbf$ denote
  the Riemannian Hessian, which are  the Euclidean gradient and Hessian followed by the orthogonal projection
 to tangent spaces, respectively. 
 Then we have
 \beqq\label{eq19}
 \mbf(\cR_X(tv))=
 \mbf(X)+
 t\langle \gd\mbf(X), v\rangle+\frac{t^2}{2}\langle\Hs \mbf(X)[v],v \rangle+O(t^3),
 \eeqq 
 See Prop. 5.44  in~\cite{boumal_introduction_2022} for the derivation details of (\ref{eq19}). 
  \begin{itemize}
\item  Employ  (\ref{eq19})
at a local optimal solution $X$ in (\ref{main_P}). Then   $\gd \mbf(X)$ must vanish, i.e.,  
 the Euclidean gradient  $\nabla \mbf(X)=AXC-B$  must lie in  $N_X St(n,r)$. Thus, 
 \beqq\label{eq23_1}
 AXC-B=X\Lambda_{sym}
 \eeqq 
 holds for some symmetric matrix $\Lambda_{sym}$.
  This  condition   is known as 
  the first-order optimal condition.
 \item  The Riemannian Hessian
   is the covariant derivative of the gradient vector,
  \begin{eqnarray}
&& \Hs \mbf(X)[v]
=\nabla_v \gd\mbf(X)\\
&=&\Pj_X \{ \nabla^2 \mbf(X) [v]-v(X^\top \nabla \mbf(X))_{sym} \}\\
&=&\Pj_X\{AvC-v\Lambda_{sym}\}.\end{eqnarray}
Observe (\ref{eq19}) at a local minimizer $X$.  The second-order condition implies that for all $v\in T_X\cM$,
 \beqq\label{eq23}
 \langle v,AvC\rangle
 -\langle v,v\Lambda_{sym} \rangle
 \ge 0.
 \eeqq
 \end{itemize}

In general, there may be many critical points $X$ fulfilling the conditions in (\ref{eq23_1},\ref{eq23}). These points are  \textit{stationary points} (maximizers, minimizers, or saddle points). The quality of these critical points $X$ is directly related to 
the eigenvalues of the associated matrix $\Lambda C^{-1}$. 
Prop.~\ref{suf_cond} gives one necessary condition to characterize local minimizers. 
This condition also works for a global minimizer of (\ref{main_P}).

\begin{prop}\label{suf_cond}
Let $X\in St(n,r)$  be a stationary point of (\ref{main_P}).
Then 
\beqq\label{1stC}
X\Lambda = AXC-B
\eeqq
 holds for some symmetric matrix $\Lambda$. 
In addition,  suppose $X$ is a local minimizer in $St(n,r)$.  
Let 
$Y:=X\Delta_0 +X_\bot\Delta_1 \in T_{X} St(n,r)$, where 
 $[X, X_\bot]\in \cO(n)$ holds for some matrix $X_\bot\in \IR^{n\times (n-r)}$.
 Then 
\beqq -\langle  Y^\top Y ,\Lambda\rangle+
\langle Y , A Y  C\rangle\ge 0\label{eq35}
\eeqq
holds.
Let $d^\bot_{min}$ be the smallest  eigenvalue of $X_\bot^\top A X_\bot$ and  let  $\gamma_1,\gamma_2,\ldots, \gamma_r$ be eigenvalues of $\Lambda C^{-1}$.
Then 
 \beqq\label{eq_36}
 d^\bot_{min}\ge \max(\gamma_1,\ldots, \gamma_r).
 \eeqq
\end{prop}

\begin{theorem}\label{Thm1} 
Let $X$ be a global  optimal solution in (\ref{main_P}),and let $\Lambda$ be its associated multiplier, 
\beqq
\Lambda=X^\top (AXC-B).
\eeqq 
Then $(X,\Lambda)$ satisfies  the first-order condition, \beqq
AXC=B+X\Lambda, \; \Lambda=\Lambda^\top,
\eeqq
and the second-order condition,
\beqq\label{d_r}
\Lambda \preceq d_r C.
\eeqq
\end{theorem}
\begin{proof}
Observe that    $d_r$ is an upper  bound for the smallest eigenvalue of $X_\bot^\top A X_\bot$, after applying
Courant-Fischer min-max theorem(Theorem 4.2.11\cite{horn_matrix_1985-1}) on Prop.~\ref{suf_cond}.  This completes the proof.

\end{proof}

\subsection{Qualified critical points } \label{sec_3.3} 

Next, we aim to further refine the upper bound condition of $\Lambda$ in Proposition ~\ref{suf_cond}, to effectively reach a global minimizer algorithmically or come close to it. To address this, we introduce the set of qualified critical points defined below. Thanks to Theorem~\ref{Thm1}, the existence of a global minimizer  guarantees the existence of these qualified critical points.

 
  
   \begin{definition} Let $C\succ 0$.
We say that $X$ is a \textbf{qualified} critical point of (\ref{main_P}), if 
\beqq
AXC=B+X\Lambda,\; \; \Lambda=\Lambda^\top,\;  \Lambda\preceq d_r C
\eeqq
hold for some  multiplier matrix $\Lambda$.
Equivalently, introduce the eigenvector decomposition,   \beqq\label{eq46}
\Gamma:=C^{-1/2}\Lambda C^{-1/2}=U\diag([\gamma_1,\ldots, \gamma_r]) U^{-1}\eeqq
via some orthogonal $U$.
The associated eigenvalues $\gamma_i$ of $\Gamma$ in (\ref{eq46}) are bounded above by $d_r$.
 
\end{definition}

Introduce another system matrix $A'$, whose eigenvalues are shifted by some scalar $\alpha\in \IR$, 
$A':= A-\alpha I_n$. 
We shall examine minimizers of  two problems with 
 \beqq\label{eq_2}
\min_{X\in \IR^{n\times r}} \left\{ \mbf'(X, A', B, C):=\langle X, A'XC\rangle-2\langle B,X\rangle\right\}
\textrm{ subject to $X^\top X\preceq  I_r$;
}
\eeqq
 \beqq\label{eq_2'}
\min_{X\in \IR^{n\times r}} \left\{ \mbf'(X, A', B, C):=\langle X, A'XC\rangle-2\langle B,X\rangle\right\}
\textrm{ subject to $X^\top X=  I_r$. 
}
\eeqq
The convex constraint set \beqq
 \cD:=\{X: \|X\|_2\le 1\}=\{X : X^\top X\preceq  I_r \}\eeqq serves as the convex hull of  $St(n,r)=\{X: X^\top X= I_r\}$.
 The  convex relaxation from (\ref{eq_2'}) to (\ref{eq_2}) prompts us to introduce   qualified critical points of (\ref{main_P}). 
 Observe that   the set of minimizers of (\ref{eq_2'}) remains invariant, regardless of  the value of $\alpha$; the set of minimizers of 
 (\ref{eq_2}) varies with respect to $\alpha$.
 The following two observations indicate that when the system matrix $A'$ does not have too many negative eigenvalues, we can provide an approximation of the solution to the second problem through the answer to the first problem. 
 \begin{itemize} 
 \item Consider the case  $\alpha\le d_1$, where
 $A'$ is positive semi-definite.  The task in  (\ref{eq_2})  is a convex minimization problem. Any local minimizer of (\ref{eq_2}) is also  a global minimizer, although the minimizer may not be located on $St(n,r)$.
When  a minimizer $X_*$ of (\ref{eq_2}) is located on  $St(n,r)$, then  $X_*$ will also be
 a global minimizer of (\ref{eq_2'}).
 \item  Suppose  $\alpha\ge d_r$. Theorem ~\ref{thm2} indicates that 
there exists  a  minimizer of (\ref{eq_2})   on  $St(n,r)$. Hence,  two problems have the same objective value, in the case with  $\alpha\ge d_r$.
 \end{itemize}

%
%

\begin{theorem} \label{thm2}Suppose $C\succ 0$.
Assume the separated eigenspace $d_r< d_{r+1}$ of $A$.  
Consider a minimizer of the relaxed problem in (\ref{eq_2}), where $A'=A-\alpha I_n$. Suppose  $\alpha\ge d_r$.
 Then 
 a  minimizer of the relaxed problem in (\ref{eq_2}) is a global minimizer of (\ref{eq_2'}).
 
\end{theorem}
\begin{proof} Since the feasible set from (\ref{eq_2'}) to (\ref{eq_2}) is enlarged, the optimal value in (\ref{eq_2}) is a lower bound for the optimal value in (\ref{eq_2'}).
We shall prove that a global minimizer $X$ in (\ref{eq_2}) satisfies $X^\top X=I_r$, which demonstrates that  the minimizer is a minimizer of (\ref{eq_2'}).

 Let $V_-$ and $V_+$ denote the matrices whose columns are unit eigenvectors of $A$ associated with the non-positive
 eigenvalues and the positive eigenvalues, respectively.
Use the decomposition to express $A$  as
$A=A_++A_-$ with $A_+:=V_+ V_+^\top A V_+ V_+^\top$ and
  $A_-:=V_- V_-^\top A V_- V_-^\top$. 
Use the decomposition to express $X,B$  as
\beqq
X=X_+ + X_-, \; B=B_++B_-,\; 
\eeqq
where $X_+, B_+\in span\{V_+\}$, $X_-, B_-\in span\{V_-\}$ and
 \beqq
B_+:=V_+ V_+^\top B,\;
B_-:=V_- V_-^\top B,\;
X_+:=V_+ V_+^\top X,\;
X_-:=V_- V_-^\top X,
\eeqq
 Fix $X_+$.
 According to
 \beqq
\min_X \mbf(X; A, B,C)=\min_{X_+}\mbf(X_+; A_+, B_+,C)+\min_{X_-}\mbf(X_-; A_-, B_-,C),
\eeqq
the optimality of $X$ indicates that   $X_-$ is a maximizer of
\beqq\label{eq44}
-\max_{Y} \mbf(Y; -A_-, B_-, C), \textrm{ subject to } Y^\top Y \preceq I_r- X_+^\top X_+=: S.
\eeqq
We may assume that $S$ has $r'$ positive eigenvalues. 
( Otherwise, we have $S=0$ and  $X=X_+\in St(n,r)$, which   complete the proof. )
 Observe that  (\ref{eq44}) is a quadratic problem of $Y S^{-1/2}$, subject to   $(Y S^{-1/2})^\top (Y S^{-1/2}) \preceq I_{r'}$.
 Since $-A_-$ is positive definite,
from the result in Prop.~\ref{max_P},  we can find a maximizer $X_- S^{-1/2}$  on Stiefel manifold. Thus,  $(X_- S^{-1/2})^\top (X_- S^{-1/2})= I_{r'}$, i.e., $X_-^\top X_- =S$, and thus $X\in St(n,r)$.

\end{proof}

\begin{rem}

Define the Lagrangian dual function associated with the primal problem in (\ref{eq_2}) as
the minimum value over $\Lambda'\preceq 0$, 
\beqq
\mbg(\Lambda')=\inf_{X\in \cD}\{ \mbf'(X)-\langle \Lambda', X^\top X-I_r\rangle\}.
\eeqq
From the KKT condition, an optimal solution
$X$ satisfies 
\beqq
A'XC-B=X\Lambda' \textrm{  for some symmetric matrix $\Lambda'\preceq 0$. }
\eeqq
Equivalently, by the definition of $A'$,
\beqq
AXC-B=X\Lambda,
\eeqq
holds for some 
 $\Lambda=\Lambda'+\alpha  C\preceq \alpha C$.
 Choose $\alpha=d_r$. Theorem~\ref{thm2} indicates that  each critical point of (\ref{eq_2}) corresponds to one qualified critical point of (\ref{eq_2'}).
 \end{rem}

\subsection{Regularized problems}
In general,   a qualified critical point of (\ref{eq_2}) is not necessarily a global minimizer. 
The following illustrates that 
a qualified critical point  $X$ in (\ref{main_P}) fulfilling 
 a tighter  condition in (\ref{2ndC})
  is a global minimizer.

\begin{theorem}[Global solutions]\label{global}
Let $X'$ be a stationary point of (\ref{main_P}) with  the associated multiplier matrix $\Lambda'$.
  Suppose  
  \beqq\label{2ndC}
d_1 C \succeq \Lambda'.
\eeqq
Then $X'$ is a global minimizer.
 In addition,  suppose  $ d_1 C \succ \Lambda'$.
Then $X'$ is the unique global minimizer.

\end{theorem}
\begin{proof}
For  $\Lambda'\in \IR^{r\times r}$ and $X\in St(n,r)$, let
\beqq\label{eq_F}
\cL(X,\Lambda' ):=
\frac{1}{2}\langle X, A X C\rangle-\langle B, X\rangle-\frac{1}{2}\langle\Lambda', X^\top X-I\rangle
.\eeqq
Reformulate (\ref{eq_F}) in terms of Taylor series of $X-X'$ around $X'$:
\begin{eqnarray}\label{eq27_1}
&&\mbf(X)=\cL(X,\Lambda')\\
&=&\cL(X', \Lambda')+\frac{1}{2}\left\{\langle (X-X') , A (X-X') C\rangle-\langle (X-X') ,  (X-X')\Lambda'\rangle\right\},\\
&\ge &\cL(X' ,\Lambda')+\frac{1}{2}\langle (X-X')^\top    (X-X'), d_1 C-\Lambda' \rangle\ge \mbf(X'), \nonumber
\end{eqnarray}
where the linear term is dropped due to  (\ref{1stC}).
Since $\Lambda'$ satisfies $d_1 C\succeq \Lambda'$, then (\ref{eq27_1}) implies that  \beqq 
\mbf(X)=\cL(X,\Lambda')\ge \cL(X' ,\Lambda')=\mbf(X')
\eeqq for each $X\in St(n,r)$, i.e., $X'$ is a global minimizer. On the other hand, suppose  $\mbf(X)=\mbf(X')$ holds for some $X\in St(n,r)$. The condition $ d_1 C \succ \Lambda'
$ implies 
the uniqueness from $(X-X')^\top (X-X')=0$.
\end{proof}

When $C=I$, the sufficient condition for a global minimizer was reported in Theorem 4.1~\cite{Zhang2007}, where authors studied the unbalanced Procrustes problem.
Generally, when $d_1\neq d_r$,  the condition in  (\ref{2ndC})  could be too strict to be fulfilled for any critical points. 
To proceed,  we consider approximate ``regularized" models, where the eigenvalues of the system matrix are lifted,  $ A\to \kA$,
\beqq\label{eq5}
\kA:=A+\sum_{k=1}^{r-1} (d_r-d_k) v_k v_k^\top=A+V_g (d_r I_r-\diag(d_1,\ldots, d_r))V_g^\top,\; V_g:=[v_1,\ldots, v_r]\in\IR^{n\times r}.
\eeqq

The following result is one specific case of Theorem~\ref{global}, where $d_1=d_r$. 
It is important to note that  in this situation, $\kA-d_r I$ is positive semidefinite and
  a qualified critical point is automatically a global minimizer, as stated in Theorem~\ref{thm2}.

 \begin{cor}\label{thm1r}
Consider the regularized problem,
\beqq\label{regu}
\min_X \{\mbf(X, \kA, B, C)=\langle X, \kA XC\rangle-2\langle X, B\rangle: X\in St(n,r)\}.
\eeqq  A qualified critical point is a global solution.
\end{cor}


We provide several additional cases, where a qualified critical point is a global minimizer. 
The following proposition serves as a preliminary example. Assume the reduced SVD: $B=U_B D_B V_B^\top$ with rank $r$.
 \begin{prop} \label{C=I}
  Suppose $C= I$. 
 Suppose left singular vectors of $B$ lie in the column space of $V_g$, i.e., $U_B=span\{V_g\}$. 
Then  
a  global minimizer $X$ in (\ref{main_P}) is
\beqq\label{eq131}
X=U_B V_B^\top=P_\cM (B).
\eeqq
 In this situation,  $X^\top B$ is symmetric and positive semidefinite, and  $\gamma_j\le d_j$ for $j=1,\ldots, r$.

 \end{prop}
 \begin{proof} The objective $\mbf$ has a lower bound: for any $X\in St(n,r)$, 
 \beqq\label{LB}
 \mbf(X; A, B, I)\ge \frac{1}{2}(d_1+\ldots+d_r)-\|B\|_*,
 \eeqq
 where $\|B\|_*$ is the nuclear norm, i.e.,  the sum of singular values of $B$. Note that the lower bound in (\ref{LB}) is reached at $X=V_g V_B^\top$, which verifies the optimality of $X$. 
 In addition,
 $\Lambda$
 satisfies  \beqq\Lambda=
 X^\top AX-X^\top B
 = V_B ( Q_B^\top  \diag([d_1,\ldots, d_r]) Q_B- D_B) V_B^\top. 
 \eeqq
Since the diagonal of $D_B$ is nonnegative, 
then $
\gamma_j\le  d_j $ for all $j=1,\ldots, r$.
 \end{proof}
 
 Motivated by the above case, we propose a non-degenerate condition on   $V_g^\top B$ to bridge the gap between the second-order conditions in Theorem~\ref{Thm1}  and Theorem~\ref{global}.
 The following safeguard estimate demonstrates that when the projection of $B$ onto   $V_g$ is sufficiently large in comparison to the spectral gap  $d_r-d_1$, any qualified critical point \( X \) will be automatically  a global minimizer in  (\ref{main_P}). 

\begin{definition}
 Let $\sigma$ be the smallest singular value of $V_g^\top B C^{-1}$.  When  $\sigma>0$, the problem in (\ref{main_P}) is said to be non-degenerate. 
\end{definition}

\begin{theorem}[Safeguard estimate] \label{safe}Let $V_g=[v_1, v_2,\ldots, v_r]\in \IR^{n\times r}$ be the ground eigenvector matrix. Let $(X,\Lambda)$ be a qualified critical point  in  (\ref{main_P}).
 Let $\sigma$ be the smallest singular value of $V_g^\top B C^{-1}$.  Then \beqq
 \gamma_j\le d_r- \sigma \textrm{ for $j=1,\ldots, r$.}
 \eeqq
  In addition,
suppose
 \beqq\label{eq285}
 \sigma>d_r-d_1.
 \eeqq
Then  $\Lambda \preceq d_1 C$ , and thus $X$ is a global minimizer.
\end{theorem}
\begin{proof}
 Start with  the first-order condition $AX=X\Lambda C^{-1}+B C^{-1}$. 
Let $u_j$ be a unit eigenvector  of $\Lambda C^{-1}$ corresponding to eigenvalue $\gamma_j$.  Taking the product of the first order condition with $u_j$ and $v_i$  from the right and left-hand sides yields 
 \beqq
d_i v_i^\top X u_j=v_i^\top  AXu_j= (v_i^\top X)\Lambda C^{-1} u_j+v_i^\top B C^{-1} u_j,
 \eeqq
which implies
 \beqq\label{eq288}
( d_i-\gamma_j) v_i^\top X u_j = v_i^\top B C^{-1} u_j.
 \eeqq
The second-order condition of a qualified point  $X$  indicates  \beqq\label{eq286}
\textrm{$\gamma_j\le d_r$ for $j=1,2,\ldots, r $.}\eeqq
Note that  $\|V_g\|_2=1=\|X\|_2$, then $\| V_g^\top X u_j\|\le \| X u_j\|\le 1$. 
Since $|v_r^\top B C^{-1} u_j|$ is bounded below by the smallest singular value of $V_g^\top B C^{-1}$, 
then with  $i=r$, 
\beqq\label{eq286'}
d_r-\gamma_j \ge  (d_r-\gamma_j)| v_r^\top X u_j | =|v_r^\top B C^{-1} u_j|\ge \sigma.
\eeqq
From  (\ref{eq285},\ref{eq286'}),
we have
 $\gamma_j< d_1$ for each $j=1,\ldots, r$. The proof is complete according to Theorem~\ref{global}.
 
\end{proof}

\subsubsection{Degenerate cases} The problem in (\ref{main_P}) can be  challenging when columns of $B$ are nearly orthogonal to $V_g$. Even for the regularized problem,
when   $V_g^\top B$ is singular, i.e., $\sigma=0$,  we encounter  the degenerate case, where 
multiple qualified minimizers $X$ can exist. 
Here is one example.
\begin{prop} Consider the regularized problem in (\ref{regu}) with
  $V_g^\top B C^{-1}=0$. Suppose
 the spectral norm bound holds:
\beqq\label{eq_76_1}
\|(\kA-d_1 I)^{\dagger} (BC^{-1})\|_2\le 1.
\eeqq
There exists  some  $U\in \IR^{r\times r}$, such that 
\beqq\label{eq_76}
X=V_g U+(\kA-d_1 I)^\dagger BC^{-1}.
\eeqq  is a global solution.
In this case, 
 the multiplier    is  $\Lambda =d_1 C$.

\end{prop}
\begin{proof} We shall construct a global solution $X$ with $\Lambda C^{-1}=d_1 I_r$.
Consider the solution in the form of (\ref{eq_76}).
For
 the condition $X\in St(n,r)$, we  choose $U$ to satisfies  \beqq\label{eq_81}
I={X}^\top X= U^\top U+ C^{-1}B^\top ((\kA-d_1 I)^\dagger)^\top   (\kA-d_1 I)^\dagger B C^{-1}.
\eeqq
The matrix $U$ is not unique.
We can determine a matrix  $U$  by Cholesky decomposition. 
We have the first-order condition 
\beqq\label{eq83}
(\kA-d_1 I) X=BC^{-1}\eeqq under  the multiplier     $\Lambda =d_1 C$.
 \end{proof}
%
Readers should be aware of the non-uniqueness of the global minimizers from the non-uniqueness of $U$ in (\ref{eq_81}).
Finally, we provide 
one example, which illustrates the existence of multiple qualified minimizers under the small $B$ compared with the spectral gap.
Roughly, computing a global minimizer becomes challenging under a small norm of $B$.
\begin{ex}
Consider the problem in (\ref{main_P}),
\beqq 
A=\left[
\begin{array}{ccc}
d_1  &  0 & 0\\
0 & d_2 & 0\\
0 &  0 & d_3
\end{array}
\right],
B=\left[
\begin{array}{cc}
\delta_1  &  0\\
0 & \delta_2\\
0 &  0
\end{array}
\right], \; C=I_2,
\eeqq
with $\delta_1>0$, $\delta_2>0$ and $d_1<d_2<d_3$.
Examine the following two critical points on $St(n,r)$: \beqq X=X_1=[e_1, e_2],\textrm{ and } X=X_2=[-e_1, e_2].\eeqq 
We have $\mbf(X_1)<\mbf(X_2)$, since
\beqq
\mbf(X_1)=(d_1+d_2)/2-(\delta_1+\delta_2),\; 
\mbf(X_2)=(d_1+d_2)/2-(-\delta_1+\delta_2).
\eeqq
The associated  
multiplier matrices are  
 $\Lambda=X^\top AX-X^\top B$,
\beqq 
\Lambda=\Lambda_1=\diag(d_1-\delta_1, d_2-\delta_2),\; \Lambda=\Lambda_2=\diag(d_1+\delta_1, d_2-\delta_2).
\eeqq  From (\ref{T_X}), the tangent space $T_X St(n,r)$ for $X=X_1$ or $X=X_2$ has a set of basis vectors,
\beqq
v_1=\frac{1}{\sqrt{2}}\left[
\begin{array}{cc}
0  &  -1\\
1 & 0\\
0 &  0
\end{array}
\right],
v_2=\left[
\begin{array}{cc}
0  &  0\\
0 & 0\\
1 &  0
\end{array}
\right],
v_3=\left[
\begin{array}{cc}
0  &  0\\
0 & 0\\
0 &  1
\end{array}
\right].
\eeqq 
From (\ref{eq23}),  the   Riemannian Hessians $H_1, H_2$ can be expressed  in terms of basis $\{v_1, v_2, v_3\}$,
\beqq
H_1
=\left[
\begin{array}{ccc}
(\delta_1+\delta_2)/2 &  0 & 0\\
0 & d_3-d_1+\delta_1  & 0\\
0 &  0 & d_3-d_2+\delta_2
\end{array}
\right],\;
H_2
=\left[
\begin{array}{ccc}
(\delta_1-\delta_2)/2 &  0 & 0\\
0 & d_3-d_1-\delta_1  & 0\\
0 &  0 & d_3-d_2+\delta_2
\end{array}
\right].
\eeqq
Since $X_1=\cP_\cM (B)$, from Prop.~\ref{C=I},
$X_1$ is
the global minimizer,  and also a qualified critical point.
On the other hand,  the optimality of $X_2$ can vary, depending on the choices of $\delta_1, \delta_2$.
\begin{itemize}
\item When $\delta_1\in (0,d_2-d_1)$, then  $ X_2$ is a qualified critical point. When $\delta_1>d_2-d_1$, $X_2$ is no longer a qualified critical point. 
\item  $ X_2$ is a local minimizer,    when $\delta_1\in (0,d_3-d_1)$ and $\delta_2\in (0,\delta_1)$, i.e.,   $H_2$ is positive definite.
\end{itemize}
\end{ex}

%
%
%
\section{Algorithms}\label{sec:3}

\subsection{Gradient projection methods }

Projected gradient methods  or gradient methods with retractions can be used 
to reach a local solution of (\ref{main_P}). 

\begin{itemize}
\item 
 When  $Y\in T_X\cM$ and $X\in St(n,r)$,   singular values of $X+Y$  are greater or equal to  $1$. Hence, 
the metric retraction $\cR_X(Y)$ is actually the projection of $X+Y$ on the convex hull of $St(n,r)$. 
The Euclidean gradient of $\mbf$ is  $
\nabla \mbf=AXC-B$. The standard  gradient projection method, i.e., 
the iterations  
\beqq\label{[]St}
X^{(k+1)}=\cP_\cM(X_k-\alpha(AX_k C-B)),  \textrm{ for $k=1,2,\ldots $, }
\eeqq
can be used to compute a stationary point
with proper  step size $\alpha>0$. The convergence analysis under the Armijo rule can be found in   Prop. 2.3.3\cite{bertsekas_nonlinear_2016}.

\item 
The Riemannian gradient descent method is  one popular method to solve (\ref{main_P}). 
The  Riemannian gradient  at $X$ is given by
\begin{eqnarray}
&&\gd  \mbf(X)=\Pj_X [AXC-B]\\
& =& (AXC-B)-X\Lambda_{sym},
\end{eqnarray}
where $\Lambda=X^\top (AXC-B)$.
The Riemannian gradient descent is 
   \beqq\label{eq_97}
     X_{k+1}= \cR_{X_k } (-\alpha_k \gd\,  \mbf(X_k));
     \eeqq
     For the gradient method  with the Armijo rule,
the convergence to a stationary point is verified in Theorem 4.3.1\cite{absil_optimization_2008} or Cor. 4.13 \cite{boumal_introduction_2022}. 

  \item When $X_k$ is near a critical point, we can use 
     the Newton method  to speed up the convergence, 
     \beqq\label{Newton}
     X_{k+1}= \cR_{X_k } (\alpha_k Z );
     \eeqq
where 
the Newton direction  $Z\in \IR^{n\times r}$ is computed  from \beqq\label{eq95}
     \Hs\,  \mbf(X)[Z]=-\gd\, \mbf (X),
     \eeqq
     The right-hand side is 
     \begin{eqnarray}
&&     \gd\,  \mbf(X)=\Pj_X \left(\nabla \mbf (X) \right)\\
&=&AXC-B-X\Lambda_{sym}, \; \Lambda=X^\top (AXC-B)
     \end{eqnarray}
and the left-hand side  (    the calculation is based on  (7.29), (7.39) in \cite{boumal_introduction_2022})
is      \begin{eqnarray}
&&     \Hs \mbf(X)[Z]=\Pj_X \left\{ \nabla^2 \mbf(X)  [Z]-Z (X^\top \nabla \mbf(X))_{sym} \right\}
\\
&=&\Pj_X \left\{  AZC-Z\Lambda_{sym}\right\}.\label{eq99}
     \end{eqnarray}



\end{itemize}

 From an optimization perspective, the Newton scheme in (\ref{Newton}) can be interpreted as a local version of Sequential Quadratic Programming  (see section 18.1\cite{nocedal_numerical_2006}).
However, when $X_k$ 
 is significantly far from a critical point, the system in equation (\ref{eq95})  may fail to be positive definite, which can render the computation of $Z$ ineffective. In particular, $Z$ may be an ascent direction. In the next section, we will propose a method for determining $\Lambda$ in regularized problems to address this issue.

\subsection{Sequential subspace methods(SSM)}

Manifold optimization involves multiple retractions at each step-length determination, which can be time-consuming, especially when $n$ is much larger than $r$. To address this, sequential subspace methods have been proposed for solving the sphere minimization problem outlined in (\ref{eq_1d'})\cite{hager_global_2005}. 
Specifically, the authors consider incorporating the Newton direction vector and the ground vector into the subspace to ensure convergence to a global minimizer for the sphere minimization problem. 
Further applications of subspace methods can be found in~\cite{CSIAM-AM-2-4}. These methods concentrate on  conducting retractions within a smaller dimensional space, leading to a significant reduction in computational complexity when dealing with large-scale problems.  In this section, we will extend the framework in \cite{hager_global_2005} to address the $St(n,r)$ minimization problem in  (\ref{main_P}), emphasizing the quality of the solutions. 
 
For a general smooth function $\mbf(X)$ on $St(n,r)$, the  idea of SSM is to introduce  a  proper isometric matrix  $V\in St(n,l)$ and
to express $X$ in the form of  \beqq
 \textrm{$X=V\widetilde X$ with $\widetilde X\in St(l,r)$, $r<l<n$ 
}\eeqq
so that 
 the original problem can be approximated by  the low-dimensional problem 
 \beqq
\min_{\widetilde X} \{ \widetilde \mbf(\widetilde X):= \mbf(V \widetilde X)\}.
 \eeqq
More precisely,  (\ref{main_P}) can be handled by the alternating scheme
  \beqq\label{eq82}
 \min_{X\in St(n,r)} \mbf(X; A, B,C)=
\min_{V} \min_{X\in St(n,r)} \{ \mbf(V^\top X; V^\top A V, V^\top B,C): V^\top V=I, \; V\in \IR^{n\times 4r}\}.
\eeqq
The matrix $V$ can be regarded as one descriptor of 
 one subspace   $\cS_k$   of   the  $n\times l$ matrices, whose columns are spanned by
\begin{itemize}
\item column vectors  of $V_g$: the ground eigenvectors  $\{v_1, v_2,\ldots, v_r\}$ of $A$ corresponding to  the smallest  eigenvalues $d_1, \ldots, d_r$,
\item column vectors of  $X_k$,
\item column vectors of  SQP  direction   $Z_k$ (the computation is deferred in section~\ref{3.4.2} ),
\item column vectors of gradient  $\nabla \mbf(X_k)=AX_k C-B$ of the cost function at $X_k$.
\end{itemize}  
The details will be given in the next section.
 In short,  
 an approximate solution  $X_{k+1}$  is chosen from a low-dimensional space
  \beqq
\cS_k=\{ [V_g, X_k, AX_k C-B, Z_k ] S: S\in \IR^{4r\times r}\}\subset \IR^{n\times r}.
\eeqq
If no confusion, we shall write $\cS_k=span\{ V_g, X_k, AX_k C-B, Z_k\}$.
Use one isometric matrix  $V=V_k$, 
from   the qr-factorization of  $[V_g, X_k, AX_k C-B, Z_k ]$ to express 
each $X\in St(n,r)\cap \cS_k$  as $X=V_k \widetilde X$ for some $\widetilde X\in St(4r, r)$.
Introduce   \beqq 
 A_k:= V_k^\top AV_k,\; B_k:=V_{k}^\top B.\eeqq
From (\ref{eq82}),   the minimizer $X_{k+1}$   of
 \beqq\label{eq110}
 \min_{X\in St(n,r)\cap \cS_k} \mbf(X)
=\min_{\widetilde X\in St(4r,r)} \mbf( \widetilde X; V_k A V_k^\top, V_k^\top B,C).
 \eeqq
 can be expressed as
  $X_{k+1}=V_k \widetilde X_k\in \IR^{4r\times r}$, where 
$\widetilde X=\widetilde X_k$ is the minimizer of the \textbf{St(4r,r) problem} induced by $S_k$,\beqq\label{eq100}
\widetilde X_k:=arg\{\min_{\widetilde X \in \IR^{4r\times r}} \mbf(\widetilde X; A_k, B_k,C )\}.
\eeqq
%
%
%

The above discussion  outlines the basic  framework of the SSM. However, when dealing  with high-dimensional data,  we need to solve many subspace problems.  A key requirement  is
to achieve fast convergence for  each  subspace problem. 
Achieving fast  convergence relies on   introducing  Newton directions in the subspace $\cS_k$. However,  if the Hessian matrix in (\ref{eq99})  is not  positive semidefinite,
the Newton step may not result in  a decrease in the objective function.
 To address  this issue, Alg.~\ref{SSM_alg} makes  two adjustments to the subspace method described by~\ref{eq110} and ~(\ref{eq110}) 
 \begin{itemize}
 \item   Replacing the objective function $\mbf$ in (\ref{eq82})  with a surrogate regularized  function $\mbf_k$, as shown in (\ref{eq128}). Here, the ground eigenvectors are included in the subspace $\cS_k$ to ensure the solution quality of each subproblem. 
\item 
 Introducing additional orthogonality to $V_g$ when calculating the Newton step.

\end{itemize}

These adjustments  bring  two  advantages. First, as stated in Corollary.~\ref{thm1r}, any qualified critical point is guaranteed to be the global minimum of every low-dimensional subspace problem in (\ref{eq130}).  Secondly, the Hessian matrix associated with the qualified multipliers is positive semi-definite, which allows us to use the conjugate gradient method to calculate the Newton step size from (\ref{eqZ}), and the resulting Newton/SQP direction can be used as a descent direction.

\subsection{Regularized surrogate models}\label{sec:Reg}
We start with the introduction of a regularized model at the base point $X_k$.
\begin{definition} 
 Let $D:=\kA-A\succeq 0$ and $C\succ 0$.
At  each base point $X_k\in St(n,r)$, introduce one  regularized  model 
$\mbf_k(X)$ of $\mbf(X; A,B,C)$,
 \beqq \label{eq128}
 \mbf_k(X)
 := \frac{1}{2}\langle X, \kA X C\rangle-\langle X, B_k\rangle+\frac{1}{2}\langle X_k, DX_k C\rangle,\textrm{ with }
  B_k=B+D X_kC.
 \eeqq
Here,    $B_k$  is introduced to ensure the identical tangent spaces of $\mbf,\mbf_k$ at $X_k$, i.e.,  the gradient condition $\nabla \mbf(X_k)=\nabla \mbf_k(X_k)$ holds, 
    \beqq\label{eq344k}
  AX_kC-B=\kA X_kC-B_k,\; i.e., B_k=B+(\kA-A) X_kC.
  \eeqq

 \end{definition}
 Observe that 
 for each  $X=X_k$ in $St(n,r)$ with $C\succ 0$, we have 
  \begin{eqnarray}
2\mbf_k(X)
&=&\langle X, A X C\rangle-2\langle X, B\rangle+\langle X_k-X, D(X_k-X) C\rangle\\
&\ge& 2\mbf(X; A,B,C).
\end{eqnarray}
Hence, as shown in 
Fig.~\ref{fig:regular-fig}, $\{ \mbf_k(X)
: X\in \cS_k
\}$ is a surrogate  model which  is tangent to  $\mbf$ at $X_k$.
Thanks to the surrogate property, instead of (\ref{eq100}), we choose  $X_{k+1}$  to be a global minimizer of the regularized problem \beqq\label{eq130}
 \min_X \{ \mbf_k(X): X\in St(n,r)\cap \cS_k \},
\textrm{ where } 
 \cS_k=span\{  X_k, A X_kC-B\}.
 \eeqq 
We have the following monotonic property: 
 \beqq 
\mbf(X_k; A,B,C)=
\mbf_k(X_{k})\ge \mbf_k(X_{k+1})\ge \mbf(X_{k+1}; A,B,C).\eeqq
   

 %
%

\begin{figure}
\begin{center}
\includegraphics[scale=0.4]{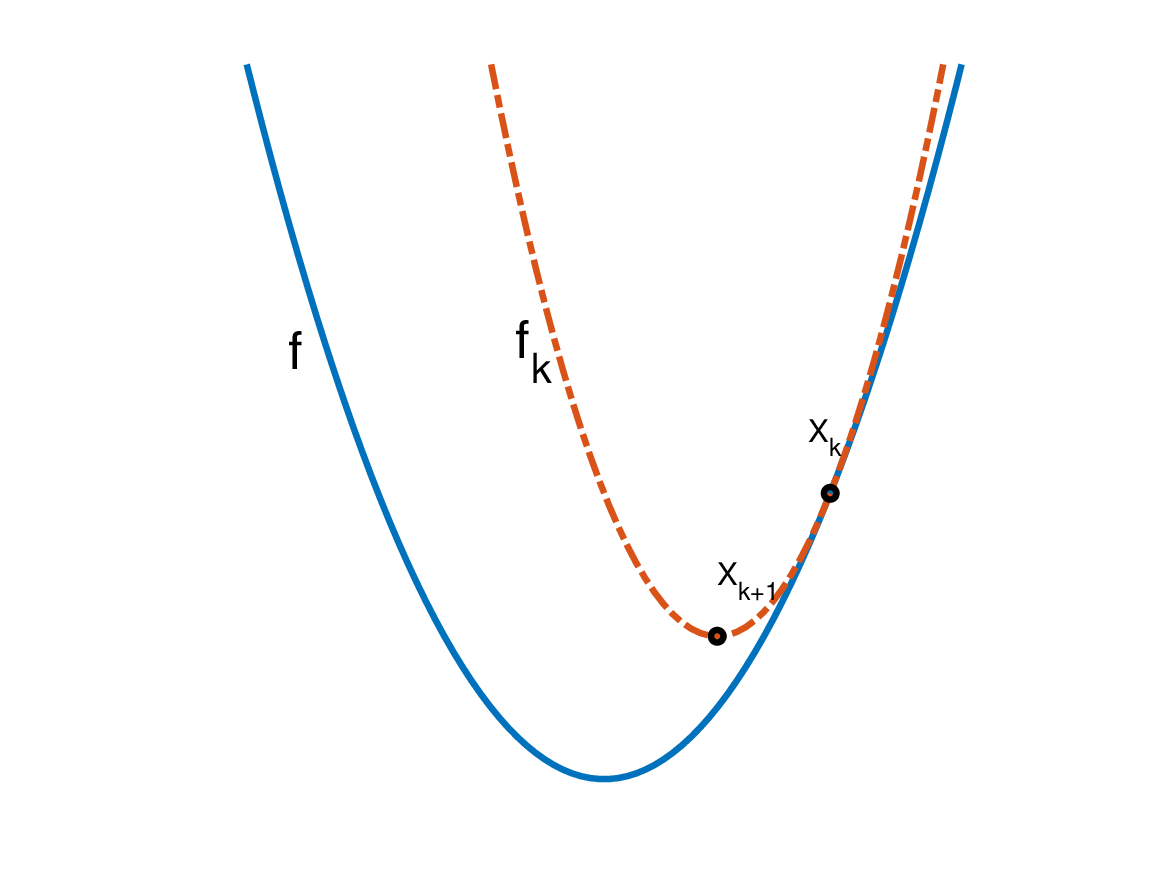}
\end{center}
\caption{Illustration of surrogate models $\mbf_k(X)$ of $\mbf(X; A,B,C)$, where $ \mbf_k(X)$ in (\ref{eq128}) is constructed from the tangent point $X_k$.   The minimizer $X_{k+1}$ of $\min_X\{  \mbf_k(X): X\in \cS_k\}$ provides one estimate for the minimizer of $\min_X \mbf(X; A,B,C)$.}\label{fig:regular-fig}
\end{figure}

 Next, we  show that the solutions to these regularized problems converge to one qualified critical point of the original problem in ~(\ref{main_P}).  The main result in Theorem~\ref{prop_3.3} is the convergence of SSM under the non-degeneracy assumption.
 
  
\begin{theorem} \label{prop_3.3} For each $k$,
let $X_{k+1}$ be a  qualified critical point  to the regularized problem in  (\ref{eq130}) with  the associated multiplier $\Lambda_{k+1}$,
\beqq\label{eqLa}
  \Lambda_{k+1}=X_{k+1}^\top (\kA X_{k+1} C-B_k).\eeqq
  \begin{itemize}
\item    Let $g_k=\gd \mbf(X_k; A, B, C)$.
Then $
\lim_{k\to \infty } g_k=0$.
\item Suppose $V_g^\top BC^{-1}$ is nonsingular.  Let  $(X_*,\Lambda_*)$ be 
 a limit point of $\{ (X_k,\Lambda_k): k\}$. 
 Then  $X_*$ is a qualified critical point of \beqq
 \min_X \mbf(X; A,B,C).\eeqq
 \end{itemize}
\end{theorem}
\begin{proof}

Introduce a curve  $c(\alpha )=\cR_X(\alpha s)$ on $St(n,r)\cap \cS$ with $s=-\|g_k\|^{-1} g_k$. The Taylor expansion around $X_k$ yields
\beqq
\mbf_k( \cR_X(\alpha s))=\mbf_k(X_k)+\langle \gd \mbf_k(X_k), \alpha s\rangle+\frac{1}{2} \langle \Hs \mbf_k(X_k)[\alpha s], \alpha s \rangle+o(\alpha^2).
\eeqq
Thanks to (\ref{eq128}), (\ref{eqLa}),    $\|B_k\|$ and $\|\Lambda_k\|$ are bounded above on $St(n,r)$.  Then  $\| \Hs \mbf_k(X)\|$ is bounded above  on $St(n,r)$.
Hence, we can find some  constant $M>0$, such that 
\beqq\label{eq138}
\mbf(X_{k+1})\le \mbf_k(X_{k+1})  \le \mbf_k(c(\alpha))\le \mbf_k( X_k)-\|g_k\| \alpha+\frac{1}{2} M\alpha^2
= \mbf( X_k)-\|g_k\| \alpha+\frac{1}{2} M\alpha^2
\eeqq
whenever $\alpha\in [0,0.5]$. Since $\|g_k\|$ is bounded over all choices of $X\in St(n,r)$,  we can choose $M$ sufficiently large  to ensure that $\|g_k\|/M\le 0.5$.
Take  $\alpha=\|g_k\|/M$ in (\ref{eq138}), and 
 we have
$
 \mbf( X_{k+1})\le \mbf( X_k)-\|g_k\|^2/(2M).
$ Thus,  $\lim_{k\to \infty} \|g_k\|=0$ holds according to 
\beqq
 \mbf( X_{k+1})\le \mbf( X_1)-\sum_{j=1}^k\|g_j\|^2/(2M).
\eeqq

By assumption,  $V_g^\top B C^{-1}$ is nonsingular. From Theorem~\ref{safe}, 
$d_r-\gamma_j\ge \sigma$ holds for some $\sigma>0$.
Expanding $\mbf_k$ around $X_k$, we have for each $X\in St(n,r)\cap \cS$,
\beqq\label{eq_99}
\mbf_k(X)=\mbf_k(X_k)+\langle X-X_k, g_k\rangle+\frac{1}{2}\langle X-X_k, \Hs \mbf_k(X) [X-X_k]\rangle.
\eeqq
Hence, 
\beqq\label{eq_100}
\frac{\sigma}{2} \|X_{k+1}- X_k\|^2\le \mbf_k(X_{k+1})-\mbf_k(X_k) -\langle X_{k+1}-X_k, g_k\rangle\le  \| g_k\| \| X_{k+1}-X_k\|.
\eeqq
Together with $
\lim_{k\to \infty} g_k=0$, we have
  $\lim_{k\to \infty} \| X_{k+1}-X_k\|=0 $.
  
 To examine the convergence to a qualified point, we  introduce    a multiplier $\wLam_{k+1}$ associated with $\mbf$, in addition to the multiplier $\Lambda_{k+1}$ associated with $\mbf$,
   \beqq \label{eq138}\wLam_{k+1}=X_{k+1}^\top (A X_{k+1} C-B),\; 
  \Lambda_{k+1}=X_{k+1}^\top (\kA X_{k+1} C-B_k).\eeqq 
As $k\to \infty$,
\beqq\label{eq129}
\wLam_{k+1}-
\Lambda_{k+1}=X_{k+1}^\top (D X_{k+1} C-(B_k-B))=X_{k+1}^\top D(X_{k+1}-X_k) C\to 0.
\eeqq
For each $k$, the global optimality of $X_{k+1}$ ensures $\Lambda_{k+1} \preceq d_r C$. Together with (\ref{eq129}),
the limit  $\wLam_*$ of $\{\wLam_{k+1}\}$ satisfies  $\wLam_*=\Lambda_*\preceq d_r C$.

\end{proof}

\subsection{Subspace selection}
To effectively calculate  qualified points,  we will incorporate $V_g$ into  the subspace $\cS$,
\beqq\label{eq147}
\min_{X\in \cS\cap St(n,r)} \left\{ \mbf_k(X)=\mbf_k(X; \kA, B_k, C )+\frac{1}{2}\langle X_k, (\kA-A) X_k C\rangle\right\}.
\eeqq
       This inclusion of $V_g$ can prevent stagnation at non-global local solutions. 
    \begin{prop}
    Suppose $\bar X$ is a stationary point of   (\ref{eq147})  and $\bar X$ is not a global minimizer. For
    $
    span\{ V_g, \bar X\}\subset \cS $, i.e., $\{ V_g U_1+\bar X U_2: \; U_1\in \IR^{r\times r}, \; U_2\in \IR^{r\times r}\}\subset\cS$, we have
    \beqq\label{eq152}
    \min_{X\in \cS\cap St(n,r)} \mbf_k(X)\le \mbf_k(\bar X). 
    \eeqq 
    \end{prop}
    \begin{proof}
    
    Let $\bar \Lambda$ be the multiplier in the regularized problem associated with the stationary point $\bar X$. Then 
    \beqq\label{mbf_k}
    \mbf_k(X)=\mbf_k(\bar X)+\frac{1}{2} \langle (X-\bar X),  \kA(X-\bar X) C \rangle-\frac{1}{2}\langle (X-\bar X)\bar \Lambda, (X-\bar X) \rangle.  
    \eeqq
    Claim: Assuming that 
    $\bar X$ is not a global minimizer, i.e., 
    $\bar\Lambda C^{-1}$ has at least one eigenvalue $\gamma_1$ that exceeds  $d_r$, if necessary, we can use  a Householder reflection to  construct $X_H \in \cS\cap St(n,r)$, which has a lower  objective value. 
    
    To reveal the structure of $\mbf_k$,  we introduce an orthogonal eigenvector matrix  $U$ to diagonalize  $\bar \Lambda$ in the right-hand side of (\ref{mbf_k}). Introduce
    \beqq \Gamma:=U^\top C^{-1/2} \bar \Lambda C^{-1/2} U=\diag(\gamma_1,\ldots, \gamma_r)\eeqq and  \beqq  Y:= XC^{1/2} U,\;
    \bar Y:=\bar XC^{1/2} U=[\bar y_1,\bar y_2,\ldots, \bar y_r].\eeqq
    Hence, we can express $\mbf_k(X)$ as a function $Y$,  
    \beqq
    \mbf_k(X)=\mbf_k(\bar X)+\frac{1}{2} \langle (Y-\bar Y),  \kA(Y-\bar Y)  \rangle-\frac{1}{2}\langle (Y-\bar Y)\Gamma, (Y-\bar Y) \rangle.  
    \eeqq
   
   Since  $V_g$ has rank   $r$,  we may choose one unit vector $v_1$ in the column space of $ V_g$, which is orthogonal to $\bar y_2,\ldots, \bar y_r$.  Starting with $v_1$, we construct a set of  unit orthogonal bases $\{v_1,v_2,\ldots, v_r\}$ of the column space of $V_g$. Together with the eigenvectors $\{v_{r+1},\ldots, v_n\}$ of $A$, we have a set of unit  orthogonal basis $\{v_1,\ldots, v_n\}$ in $\IR^n$. Write $\bar y_1=\sum_{j=1}^n \xi_j v_j$ for some scalars $\xi_j$.
To proceed, we will consider two cases.   
 \begin{itemize} 
 \item  Suppose $\xi_1\neq 0$.
   Take $Y_H=[(y_1)_H,\bar y_2,\ldots, \bar y_r]$ with 
    \beqq
   (y_1)_H=\bar y_1-2 \xi_1 v_1=-\xi_1 v_1+\sum_{j=2}^n \xi_j v_j.
   \eeqq 
   The transformation  $\bar Y\to Y_H$ can be regarded as  one Householder reflection about the hyperplane orthogonal to $v_1 e_1^\top$, where $e_1=[1,0,\ldots, 0]^\top \in \IR^r$.
   Let  $X_H=Y_H U^\top C^{-1/2}$.
   Clearly,   $X_H\in St(n,r)$, since
   \begin{eqnarray}
 &&X_H^\top X_H=(Y_HU^\top C^{-1/2})^\top(Y_HU^\top C^{-1/2})\\
 &=& (\bar YU^\top C^{-1/2})^\top(\bar YU^\top C^{-1/2})=\bar X^\top  \bar X=I,
    \end{eqnarray}
    where we used
   \beqq
    Y_H^\top  Y_H=(\bar Y-2\xi_1 v_1 e_1^\top )^\top (\bar Y-2\xi_1 v_1 e_1^\top )=\bar Y^\top \bar Y.
    \eeqq 
    Observe that $X_H-\bar X=(Y_H-\bar Y) U^\top C^{-1/2}=-2\xi_1 v_1 e_1^\top  U^\top C^{-1/2}$ implies   $ X_H\in \cS$. Thus (\ref{eq152}) holds, 
    \begin{eqnarray}
&&    \mbf_k(X)-\mbf_k(\bar X)=\frac{1}{2} \langle (X-\bar X),  \kA(X-\bar X) C \rangle-\frac{1}{2}\langle (X-\bar X)\bar \Lambda, (X-\bar X) \rangle\\
&=&\frac{1}{2} \langle -2\xi_1 v_1 e_1^\top ,  A(-2\xi_1 v_1 e_1^\top   )  \rangle-\frac{1}{2}\langle (-2\xi_1 v_1 e_1^\top  )\Gamma, (-2\xi_1 v_1 e_1^\top ) \rangle\\
&\le &2\xi_1^2 (d_r-\gamma_1)<0.
    \end{eqnarray}
    
  \item   Suppose  $\xi_1=0$, i.e., $v_1$ is orthogonal to  all the vectors $\bar y_1,\ldots, \bar y_r$.  Then we have  \beqq
    V'=v_1 e_1^\top U^\top C^{-1/2} \in T_{\bar X} St(n,r).\eeqq 
    Indeed, $\Pj_{\bar X}(V')=V'-\bar X (\bar X^\top V')_{sym}=V'$,
    where ${V'}^\top \bar X ={V'}^\top \bar Y  U^\top C^{-1/2}=0$. Under this circumstance, we can show that $\bar X$ is actually  not a local minimizer of  (\ref{eq152}). Consider a smooth curve $c(\alpha)\in St(n,r)$ passing through $c(0)=\bar X$ and $c'(0)=V'$.
    Then  we have (\ref{eq152}), since
    for sufficiently small $\alpha>0$, from (\ref{eq19}),
    \begin{eqnarray}
&&    \mbf_k(c(\alpha))=\mbf_k(\bar X)+ \alpha^2 \langle V', \Hs \mbf_k(\bar X) V'\rangle +O(\alpha^3)\\
&=&\mbf_k(\bar X)+\alpha^2 (d_r-\gamma_1) +O(\alpha^3)< \mbf_k(\bar X).
    \end{eqnarray}
    \end{itemize}
         \end{proof}
%
%

  \subsubsection{SQP direction   $Z_k$}\label{3.4.2}
  To accelerate convergence of SSM in solving (\ref{eq147}), we will  judiciously incorporate the Newton step (SQP direction) $Z_k$  into the subspace $\cS_k$.  This strategic approach is designed to ``maximize the reduction of the low-dimensional regularized model" $\mbf_k(X):=\mbf_k(X; \kA,B_k,C)$ within the tangent space at $X_k\in St(n,r)$.
    \begin{prop}[SQP direction]
  Let $X_{k+1}$ be the minimizer  of the unconstrained objective function 
 \beqq
arg\min_{X\in \IR^{n\times r}} \left\{ 
\mbm(X):=\mbf_k(X)-\frac{1}{2}\langle \Lambda_k, X^\top X  -I\rangle\right\},
\eeqq
where $\Lambda_k$ is symmetric. 
Introduce some  $Z_k\in \IR^{n\times r}$ to
express the minimizer $X$ as 
\beqq
X=X_k +Z_k+V_g u_k
\eeqq
for   some $u_k$ in $\IR^{r\times r}$. 
Assume the orthogonality between the column space of $Z_k$ and the column space of  $V_g$. 
Let $P$ denote the projection \beqq 
P=I-V_{g} V_{g}^\top.\eeqq 
 Then $Z_k$
satisfies $P Z_k=Z_k$ and
 \beqq\label{eqZ}
    P \kA P Z_k C- Z_k\Lambda_k =PE_k,\; \textrm{ where }  E_k:=-\kA X_k C+B_k+X_k\Lambda_k.
    \eeqq
  
\end{prop}

\begin{proof}
Express the minimizer $X$ as 
$
X=X_k+z
$
for some $z\in \IR^{n\times r}$. Expanding the function $\mbm$ around $X_k$, introduce $\mbm_k(z):=\mbm(X)$. 
Thanks to 
the computation
    \begin{eqnarray}
    \mbm_k(z)
&=&\frac{1}{2}\langle \kA(X_k+z), (X_k+z) C\rangle-\langle X_k+z, B_k\rangle-\frac{1}{2}\langle \Lambda_k, (X_k+z)^\top (X_k+z) -I\rangle\\
&   =&\mbm_k(0 )+\langle z, \kA X_k C-B_k-X_k\Lambda_k\rangle+
\frac{1}{2}(\langle \kA z C,z\rangle-
\langle \Lambda_k, z^\top z \rangle)\\
&   =&\mbm_k(0 )-\langle z, E_k \rangle+
\frac{1}{2}(\langle \kA zC,z\rangle-
\langle \Lambda_k, z^\top z \rangle),\label{eq429}
    \end{eqnarray}   
    the minimizer   $z$  satisfies
\beqq\label{eq145}
\kA z C-z\Lambda_k=E_k.\eeqq
Decompose  $z$ as $z= u'+Z_k$ for some $u'$ in the space of $V_g$ and $Z_k$ in the space orthogonal to $V_g$.
    Applying $P$ on the both sides of (\ref{eq145}), we have the equation of $Z_k$ in (\ref{eqZ}).

\end{proof}
\begin{rem}\label{rem3.3} In simulations,  calculating the Newton step is the most computationally intensive part of  the entire algorithm. The orthogonality between $Z_k$ and $V_g$ is introduced to ensure one positive definite system in (\ref{eqZ}), provided that   
 $\Lambda_k \preceq d_r C$ holds at each qualified point. 
 This approach allows us to use conjugate gradient methods to determine a Newton direction, which facilitates the rapid convergence of the SSM.
  \end{rem}

\subsubsection{Initialization in  SSM}
 The following discusses an effective initialization of $(X,\Lambda)$ in SSM.
 Observe one  lower bound for $\mbf(X; \kA, B, C)$: for any $X\in St(n,r)$, 
 \beqq\label{LB}
 \mbf(X; \kA, B, C)\ge \frac{1}{2}(tr(C)d_r)-\|B\|_*.
 \eeqq
The lower bound is attained if columns of $B$ lie in $V_g$ and $X=P_\cM (B)$. 

\begin{prop}[Initialization] \label{2.13}
Consider 
the approximate  regularized problem \beqq\label{eq72}
\min_X \{ \mbf(X; \kA, V_g V_g^\top B,C): X\in St(n,r)\}.
\eeqq
Then $X=P_\cM( B)$ is a global minimizer and 
 the multiplier $\Lambda=X^\top (\kA XC-V_g V_g^\top B)$ satisfies  $\Lambda \preceq d_r C$.
\end{prop}
\begin{proof}Consider $X$ in the form of   $X=V_g Q$, where
 $Q$ is one orthogonal matrix. The condition of  maximize $\langle Q, V_g^\top B  \rangle$ is the polar decomposition, i.e.,   $Q^\top V_g^\top B $ is symmetric and positive semidefinite (See \cite{horn_matrix_1985-1}). 
Since the lower bound in (\ref{LB}) is reached, we have the optimality of $X$.
 Note that 
\beqq
 \Lambda=X^\top (\kA XC-B)=d_r C-Q^\top V_g^\top B 
\eeqq
is symmetric and $\Lambda\preceq d_r C$.
Hence, 
 $\Lambda$ can be expressed as the difference between two Hermitian  matrices, and by 
 Weyl's inequality, the proof is complete.
 \end{proof}

     \subsubsection{ Outline of Algorithm}
Based on the above discussion, we propose the following  algorithm to compute a qualified critical point of (\ref{main_P}).

\begin{algorithm}[H]\label{SSM_alg}
  \KwData{
  Symmetric matrix $A\in \IR^{n\times n}$ with spectral gap $d_r<d_{r+1}$. Let $  B\in \IR^{n\times r}$ and  $ C\in \IR^{r\times r}$ with $C\succeq 0$. Stopping criterion parameters: $\epsilon>0$ and $k_{\max}\in \IZ$.   \\ Construct $V_g=[v_1,\ldots, v_r]$.  Use $V_g$ to construct  $\kA$ from (\ref{eq5}). 
  }
  \KwResult{  A qualified critical point $X_k\in \cM:=St(n,r)$.}
  Initialization:  $k=1$, $X_1=\cP_\cM (V_g V_g^\top B)$.
      $\Lambda_1=X_1^\top \kA X_1 C-X_1^\top B\preceq d_r C$.
  \\
  \While{$\|\gd \mbf(X_k; A,B,C)\|>\epsilon$ or $k\le k_{\max}$}{ 
 Use conjugate gradient methods to compute  $Z_k$  in (\ref{eqZ}). \\
            Let  $\cS_{k}=span\{ V_g, X_{k}, AX_{k} C-B,Z_k\}$ and
let  $V_{k}\in \IR^{n\times 4r}$ be the isometric matrix associated with  $\cS_{k}$. 
        Compute $B_k$ from (\ref{eq344k}) and $\kA_k,\widetilde B_k$ from (\ref{eq159}).  \\
        Let $X_{k+1}=V_k \widetilde X$, where $\widetilde X\in \IR^{4r\times r}$ is the solution of $St(4r,r)$ problem in   (\ref{eq170}) via  Alg.~\ref{SSMcore}.
         \\
       $k\to k+1$.
         }
  \caption{Sequential subspace methods to solve (\ref{main_P})}
\end{algorithm}

{
\begin{rem}[Approximate qualified critical points]\label{rem3.5}
Various algorithms can be used to solve the low-dimensional problem in  (\ref{eq170}). Our simulation uses 
 Riemannian Newton methods described in the next subsection. 
 In actual floating-point operations, when we use Algorithm~\ref{SSMcore} to solve the SSM subproblem, we can exit the while-loop once we reach a suitable approximation point. That is,  we take $Y_j$ as   an approximate qualified critical point $X_{k+1}=Y_j$, if    \beqq\label{eq:eq133}
\mbf(X_k)-\mbf(X_{k+1}) \ge \mbf_k(X_k)-\mbf_k(X_{k+1})\ge c\| \gd \mbf_k(X_k)\|^2=c\| \gd \mbf(X_k)\|^2
\eeqq 
holds for  a constant $c>0$ (independent of $k$),\footnote{The inequality  can be  obtained  under proper line search. For instance, see section 4.5 in \cite{boumal_introduction_2022}. }  
 and the multiplier $\Lambda_{k+1}:=X_{k+1}^\top (\widetilde A X_{k+1} C-B_k)$ satisfies the cone condition
  \beqq (\Lambda_{k+1})_{sym}\preceq d_rC.\eeqq 
  With slight adjustments, the convergence to a qualified critical point can still be obtained based on the proof of Theorem~\ref{prop_3.3}.
    The main difference is that the current multiplier $\Lambda_{k+1}$ is generally not symmetric.
   Here, the positive constant $c$ (independent of $k$) is a sufficiently small value,  depending on an upper bound of $\|\Hs \mbf(X)\|$ on $St(n,r)$. 
 Telescoping inequalities in (\ref{eq:eq133}) for  $k$ gives
 the convergence   $
\lim_{k\to \infty} \gd \mbf(X_k)=0$.  Let $X_*$ be a limit of $\{X_k: k\}$. Thus, $\gd \mbf(X_*)=0$ and
\beqq
AX_* C-B=X_* \Lambda_*
\eeqq
holds for some symmetric matrix $\Lambda_*$.
With $\sigma>0$, introducing $\widetilde \Lambda_{k+1}$ in (\ref{eq138}) and
applying the arguments used from (\ref{eq_99}) to (\ref{eq129}),  we have  $\lim_{k\to \infty} \|X_{k+1}-X_k\|=0$ and 
 $\lim_{k\to \infty} \|\widetilde \Lambda _{k+1}-\Lambda_{k+1}\|=0$. 
Thus $\Lambda_*:=\lim_{k\to \infty} \widetilde \Lambda_{k+1}=\lim_{k\to \infty} \Lambda_{k+1} \preceq d_r C$. Thus, $X_*$ is a qualified critical point. 
\end{rem}
}

%
%
%
%
%
%

\subsection{$St(4r,r)$-problems}\label{dual}
%
%
%
%
%
%

 In this section, we shall describe   the  computation of  one qualified critical point $X=X_{k+1}$ of 
 the regularization problem in (\ref{eq147}).
Let \beqq \label{eq159}
\kA_k=V_k^\top  \kA V_k, \widetilde B_k=V_k^\top B_k.
\eeqq The update of the base point is
given by $X_{k+1}=V_k \widetilde X$ in (\ref{eq147}), where  $\widetilde X$ is  the solution to the \textbf{ St(4r,r)-problem}
\beqq\label{eq170}
\min_{\widetilde X} \left\{\mbf(\widetilde X; \kA_k, \widetilde B_k, C): \; \widetilde X\in St(4r, r)\right\}.
\eeqq

\subsubsection{Riemannian Newton methods}
Suppose   $X_{k+1}$ is in the proximity of  $X_k$.  The  Riemannian Newton method stated in Alg.~\ref{SSMcore} is a fast algorithm to solve   the  low-dimensional  \textbf{ subspace problem}, 
\beqq\label{eq185} X_{k+1}=arg\min_X \mbf_k(X; \kA, \widetilde B_k, C),\eeqq 
where $\mbf_k$ is the regularized model based at $X_k\in St(4r,r)$. Let $\Lambda_{k+1}$ be
the associated multiplier with  $ \Lambda_{k+1}\preceq
d_r C$.

 To avoid the subscript/notion confusion with previous discussions, we drop the tilde notion and the subscript $k$ in (\ref{eq185}), and replace $(X,\Lambda)$ with $(Y,\Xi)$. Write (\ref{eq185}) as 
 \beqq\label{eq185'} Y^*=arg\min_Y \mbf(Y; \kA,B, C) \in St(4r,r).\eeqq
 The optimal condition is that 
 \beqq\label{eq178}
 \kA Y^* C-Y \Xi^*=B
 \eeqq
 holds for some symmetric multiplier $\Xi^*$ with $\Xi^*\preceq d_r C$.
 The optimal solution $Y^*$ is expected to be near   the current base point $X_k$.  
The  Riemannian Newton method uses a Newton direction $Z=Z_j$ to iterate $Y=Y_j$, 
\beqq
Y_j\to Y_{j+1}=\cR_{Y_j} (\alpha Z_j).
\eeqq
 To impose the qualified condition  during the iterations,
 it is convenient to introduce the following safeguard mapping  $\mbs(\Xi)$,
according to  Prop.~\ref{safe}.

\begin{definition}Let $\sigma$ be the smallest singular value of  $V_g^\top \widetilde B_k C^{-1}$.  Suppose $\sigma>0$.
Introduce the following safeguard mapping  $\Lambda\to \mbs(\Lambda)$, so that $\mbs(\Lambda)\preceq d_r C$. 
Diagonalize   $\Lambda$ as \beqq
C^{-1/2}\Lambda C^{-1/2}=U\Gamma U^{-1},\; 
\Gamma=\diag(\gamma_1, \ldots, \gamma_r).
 \eeqq 
Update 
\beqq
\Lambda\to  \mbs(\Lambda):=C^{1/2} U\diag(\widetilde\gamma_1,\ldots, \widetilde\gamma_r) U^{-1} C^{1/2},
\eeqq
where
 \beqq
     \widetilde \gamma_i=\min(\gamma_i, d_r- \sigma),\; i=1,\ldots, r.
    \eeqq 
\end{definition}

\begin{rem}
The equation in (\ref{eq190}) is the linearization of (\ref{eq178}) with $Y^*=Y_j+Z_j$.
Thanks to the safeguard $\mbs$ on $\Xi_j$, the system in (\ref{eq190}) is positive semi-definite in $Z$. Thus, $Z_j$ can be determined by conjugate gradient methods. 

\end{rem}

\begin{algorithm}[H]\label{SSMcore}
  \KwData{
  Symmetric matrix $A=\kA_k \in \IR^{4r\times 4r}$ with spectral gap $d_r<d_{r+1}$. $  B=\widetilde B_k\in \IR^{4r\times r}$ from $V_k$ and (\ref{eq159}).   $ C\in \IR^{r\times r}$ with $C\succeq 0$. Stopping criterion parameters: $\epsilon>0$ and $j_{\max}\in \IZ$.  
  }
  \KwResult{  A qualified critical point $\widetilde X=Y_j\in \cM:=St(4r,r)$.}
  Initialization:  $j=1$, $Y_1=X_k$.
      $\Xi_1=\Lambda_k\preceq d_r C$.
  \\
  \While{$\| E_j\|>\epsilon$ or $j\le j_{\max}$}{ 
  Compute $
E_j=-A Y_j C+B+Y_j\Xi_j.$\\
 Compute $Z=Z_j$ in $T_{Y_j} St(n,r)$ from 
\beqq\label{eq190}
\Pj_{Y_j}\{A  (\Pj_{Y_j} Z) C- (\Pj_{Y_j}  Z)\mbs(\Xi_j) \}=\Pj_{Y_j} E_j.
\eeqq\\
Use the step-length search:  \beqq 
Y_{j+1}=P_\cM(Y_j+\alpha  Z_j)
\eeqq for some $\alpha>0$, which meets the sufficient decrease condition of $\mbf$ in (\ref{eq185'}).\\
Update $
\Xi_{j+1}=({Y_{j+1}}^\top (A Y_{j+1}C-B))_{sym}
$, and
     $j\to j+1$.
  }
  \caption{Riemannian Newton methods to  solve   \textbf{ St(4r,r)-problem}
 in (\ref{eq170})}
\end{algorithm}

 \section{Numerical Simulations}

 We present one experiment on  the $r$-way classification with partial labeling to illustrate 
the effectiveness of SSM.

\subsection{Spectral graph embedding  with partial labeling}\label{sec4.2}

Introduce the graph $\cG=(\cV,\cE,\cW)$, where
 $\cV = \{v_1, v_2,\ldots , v_{m'}\}$ is  the  vertex set,  $\cE$ is the edge set, and $\cW$ is the weight matrix  
with non-negative  entries $w_{ij}$ representing the weights between vertices  $v_i$ and $v_j$. 
Assume that the graph is symmetric, i.e., $w_{ij} = w_{ji}$.  
For each vertex $v_i$, 
define the degree $w_i = \sum_{j=1}^{m'} w_{ij}$.
For each subset $\cS$ in $\cV$, we can compute  the volume $vol(\cS)$ of $\cS$ and
 the cut 
$cut(\cS)$,
\beqq
vol(\cS)=\sum_{i\in \cS} w_i,\; 
cut(\cS)=\sum_{\{i,j\}\in \cE(\cS)} w_{ij},\; \cE(\cS):=\{\{i,j\} : v_i\in \cS, v_j\in \cS^c \},
\eeqq
where $\cS^c$ denotes  the complement of $\cS$. 
The conductance $\phi(\cS)$\cite{chung_spectral_1997}\cite{Chung2007}  is defined as
\beqq\label{eq182}
\phi(\cS):=\frac{cut(\cS)}{ \min\{ vol(\cS),vol(\cS^c)\}}. 
\eeqq
In the context of $r$-class classification, our goal is to   find an optimal  partitioning $\cS_1\cup\cS_2\cup\ldots\cup \cS_r=\cV$ that minimizes the conductance,
\beqq \label{eq_5}
\min_{\cS_i} \sum_{i=1}^r\phi(\cS_i).
\eeqq
For each vertex $v_i$,  we  assign a label $y_i$ from  the set of  $r$ standard basis vectors $\{e_1, e_2,\ldots , e_r\}$. 
Each standard basis vector $e_i$ is of the form  $[0,\ldots 0,1,0,\ldots, 0]$,  a one-hot row vector. 
Consequently,  the set $\cS_i$ consists of the vertices that share the same  label $y_i$.

 For simplicity, we ignore the denominator of $\phi(\cS)$, 
 let $c:=[c_1, \ldots, c_r]$ denote the cardinality vector, and 
 consider the cardinality-fixed minimization problem
\beqq \label{eq_6}
\min_{\cS_i} \sum_{i=1}^r cut(\cS_i),  \textrm{ subject to } |\cS_i|=c_i\in \IZ,\; i=1,\ldots, r.
\eeqq
The optimal partitioning (labeling) in (\ref{eq_6}) can be described by 
 the binary matrix  
\beqq
Y:=
\left[
\begin{array}{c}
y_1   \\
 \vdots  \\
  y_{m'}  
\end{array}
\right] \in \IZ^{m'\times r},
\eeqq
where  the $i$-th column of $Y$ is the indicator vector of $\cS_i$, and the column sum of   $Y$
is  $[c_1,\ldots, c_r]$.
The spectral approximation of  (\ref{eq_6}) can be computed via relaxing the min-cut problem  (\ref{eq_6}) over binary matrices to  real matrices. 
 Let  $L$ denote  the combinatorial graph Laplacian,  \beqq
 L=\diag([w_1,\ldots, w_{m'}])-\cW.
 \eeqq
An embedding of the labeling  of  vertices  is given by  the eigenfunctions $X$  corresponding to the smallest nontrivial eigenvalues, 
\beqq\label{224_'}
\min_{X\in \IR^{m'\times r}}\langle X, L X\rangle,\; X^\top X= \diag( c).
\eeqq

Consider the  semi-supervised learning task, where  a set of pre-specified labeled data is available:
The first $m$ vertices $\cV_l:= \{v_1, v_2,\ldots , v_m \}$ are assigned labels $\{y_1, y_2,\ldots , y_m\}$, where $0 < m \ll m'$.
  Let $n$ denote the number of unlabeled vertices, $n = m'-m$. The semi-supervised learning task is to smoothly propagate the labels over the unlabeled vertices $\cV_u := \{v_{m+1}, v_{m+2},\ldots , v_{m'}\}$.   Write \beqq\label{LX} 
L = \left[\begin{smallmatrix}
L_{l,l} & L_{l,u}\\
L_{u,l} & L_{u,u}
\end{smallmatrix}\right],\; Y=
 \left[\begin{smallmatrix}
Y_{l} \\
Y_{u}
\end{smallmatrix}\right],\; X=\left[
\begin{array}{c}
X_l\\
X_u
 \end{array}
\right].
\eeqq
where subscripts $l$ and $u$ correspond to labeled and unlabeled indices, respectively. 
 Let  $X_l$ represent the \textbf{labeled vertices} with cardinality $m$, i.e., $X_l=Y_l$. 
 Observe that $X$ lies in 
  the   constraint set
\beqq\label{cX}
\cX:=\{X\in \IR^{n\times r}: 1_n^\top X=c^\top, X^\top X=X_l^\top  X_l+X_u^\top X_u =  \diag(c)
\}.
\eeqq
The following proposition states that the unknown matrix  $X_u$ can be computed using a quadratic minimization problem, i.e.,  (\ref{eq189}). With additional algebraic substitutions, SSM can be utilized to solve for $X_u$, as elaborated in Remark~\ref{last}.

\begin{prop} \label{prop 1.1}
Introduce   $c,\cX$ and $ X,L$  as described above. Let $X_l\in \IR^{m\times r}$ be a binary matrix. Suppose $X$ is 
 a minimizer of 
 \beqq\label{224_'}
\min_{X\in \IR^{n\times r}}\{ \langle X, L X\rangle:  X\in \cX\}.
\eeqq
Let    $c_u=Y_u^\top 1_n$.  Introduce 
$Z_0=n^{-1}1_{n}  c_u^\top$ and
$ P= I_n-n^{-1} 1_{n} 1_{n}^\top$. Let 
  \beqq \label{AB'}
 A=PL_{u,u} P,\; B=P(  L_{u,u} Z_0+ L_{u,l} X_l),
 C=\diag(c) -X_l^\top X_l -Z_0^\top Z_0.
 \eeqq
 Then  
 $X_u$ is given by 
\beqq \label{eq191}X_u=Z_u  +  Z_0,\; Z_0:=n^{-1}1_{n}  c_u^\top,\eeqq
where $Z_u$ is a minimizer of 
 \beqq\label{eq189}
 \min_{Z}
 \left\{ \mbf(Z):=\langle Z, A Z \rangle+
 2 \langle Z , B  \rangle: \; Z^\top Z=C,\; 1_n^\top Z=0\right\}.
 \eeqq
\end{prop}
\begin{proof}
Introduce  $ X_u':=X_u- Z_0$ to
reformulate  (\ref{224_'}). From (\ref{cX}),
  $X_u'$  satisfies  the constraints:   $1_n^\top X_u'=0$, and
 \beqq
 \diag(c)=X^\top X=X_l^\top X_l+ {X_u'}^\top X_u'+Z_0^\top Z_0.
 \eeqq
 Introduce  $A':=L_{u,u}$ and $B':= L_{u,u} Z_0+ L_{u,l} X_l$. Then 
  $X_u'$ is a minimizer of 
    the problem
 \beqq
 \min_{X_u'}
 \left\{ \langle X_u', A' X_u'\rangle+
 2 \langle X_u', B'\rangle:
 1_n^\top X_u'=0,  \; {X_u'}^\top X_u'=C
 \right\}.
 \eeqq

The proof is completed by expressing $X_u'$ as $X_u'=PZ_u $.
\end{proof}
 
 \begin{rem} \label{last} 
 We want to highlight two observations.  First, $L$ possesses a null vector $1_{m'}$, which does not provide any  geometric information about the data graph.  Therefore, we decompose $X_u$ into  $X_u=X_u'+Z_0$, ensuring the  columns of $X_u'$ are orthogonal to $1_n$.  Second, 
we can   effectively enforce the constraint $1_n^\top Z=0$ during  the computation of $Z$ by selecting  the subspaces in SSM  to be orthogonal to $1_n$.  To satisfy this constraint, we incorporate  the projection $P$ in the formation of $A$ and $B$, as shown  in (\ref{AB'}).  However, 
 we cannot directly employ SSM to compute $Z_u$ based on  Prop.~\ref{prop 1.1}, since $C$ is not positive definite. 
 Indeed,  since $Y^\top Y=\diag(c)$ and $Z_0=n^{-1}1_n Y_u^\top  1_n^\top $, then 
 \beqq
 C=Y^\top Y-X_l^\top X_l-Z_0^\top Z_0=Y_u^\top Y_u-Z_0^\top Z_0
 =Y_u^\top Y_u-Y_u^\top  (n^{-1}1_n 1_n^\top) Y_u=Y_u^\top P Y_u \succeq 0.
 \eeqq
Besides,  since $C$ has a null vector $1_r$,
  the matrix $C$ in (\ref{AB'})  is positive semidefinite with  rank at most  $r-1$. 
    To proceed, we make further substitutions to achieve the standard form for SSM, i.e., (\ref{main_P}).
Let us assume that the rank of $C$ is $r'$. Consider the eigen-decomposition of $C$, \beqq
C=\widetilde Q \widetilde C \widetilde Q^\top,\;\textrm{ with isometric } \widetilde Q\in \IR^{r\times r'}, \textrm{ diagonal }\widetilde C\in \IR^{r'\times r'}.
\eeqq
Introduce  \beqq \widetilde Z:=Z_u \widetilde Q \widetilde C^{-1/2},\; \widetilde B:=B \widetilde Q\in \IR^{n\times (r-1)}
  \eeqq  to eliminate $Z_u$ in (\ref{eq189}).  Then $\widetilde Z$ is a minimizer of the problem
   \beqq\label{228_}
 \min_{\widetilde Z}
 \left\{ \langle \widetilde Z, A \widetilde Z \widetilde  C\rangle+
 2 \langle \widetilde Z ,  \widetilde  B \widetilde C^{1/2}\rangle: \; \widetilde Z\in St(n,r')\right\}.
 \eeqq
Employing SSM on the quadratic minimization in (\ref{228_}) to compute $\widetilde Z$, we can obtain 
\beqq\label{eq_199}
X_u=\widetilde Z \widetilde C^{1/2}\widetilde Q^\top  +Z_0.
\eeqq
\end{rem}

\subsection{Experimental setup}

We evaluate SSM on synthetic and three image datasets: MNIST \cite{Deng2012}, Fashion-MNIST \cite{Xiao2017} and CIFAR-10 \cite{Krizhevsky2009}. On CIFAR-10 we preprocess images using a pre-trained autoencoder as a feature extractor. The autoencoder architecture, loss, and training, were derived from the AutoEncodingTransformations architecture from \cite{Zhang2019}, with all the default parameters from their paper, and we normalized the features to unit vectors. We also evaluate SSM on a variety of real-world networks including the Cora citation network, which demonstrates that our method generalizes beyond $k$-NN graphs. In this benchmark,  the vertices of the graph represent documents and their links refer to citations between documents, while the label corresponds to the topic of the document.

\subsubsection{Synthetic datasets}\label{4.2.1}

To illustrate the concept of semi-supervised learning, we apply SSM to a small synthetic benchmark. We consider three concentric circles centered at the origin with radii $1$, $2$, and $3$ for each circle, we randomly sample $2000$ points, normally distributed with standard deviation $0.2$. We construct a graph over the data. The graph was constructed as a $k$-nearest neighbor graph with Gaussian edge weights given by
$w_{ij} =\exp\left( -4||x_i-x_j||^2/d_k(x_i)^2 \right)$ where $d_k(x_i)$ is the distance between $x_i$ and its $k^{\rm th}$ nearest neighbor. We used $k=10$ in all experiments and symmetrize $W$ by replacing $W$ with $\frac{1}{2}(W+W^\top)$. From each circle, we uniformly randomly select $5$ points to use as supervision and then construct matrices $A$, $B$, and $C$ as stated in Prop.~\ref{prop 1.1}  and run SSM to get predictions. Figure~\ref{fig:circles}  illustrates this procedure. Since there are three classes and five labeled points are selected from each class, 
the total number of labeled vertices in this synthetic experiment is \(m=15\).

\begin{figure}[th!]
\centering
\includegraphics[width=0.3\linewidth]{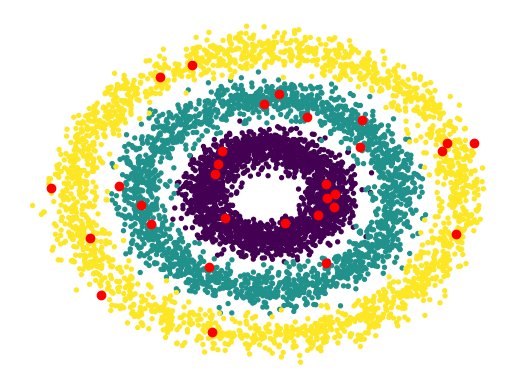}
%
\includegraphics[width=0.3\linewidth]
{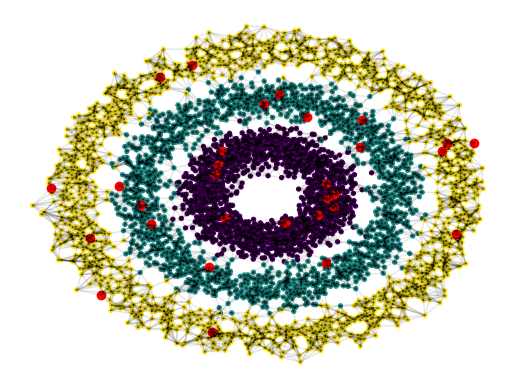}
%
\includegraphics[width=0.3\linewidth]
{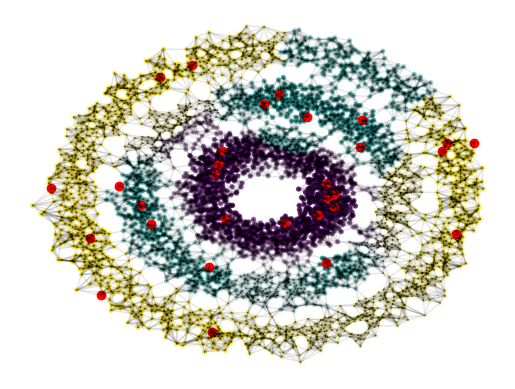}
%
\caption{Left to right: 
(a) Data distribution colored by class. 
(b) Data distribution colored by class, including edges of the $k$-NN graph with $k = 10$. 
(c) Data distribution colored by model predictions. The shade of the color is given by the norm of the associated row in $X$.}
\label{fig:circles}
\end{figure}

\subsubsection{Image datasets} 
We construct a graph over the pixel space.
 We used all available data to construct the graphs for MNIST, Fashion-MNIST, and CIFAR-10. The graph was constructed as a $k$-nearest neighbor graph with Gaussian edge weights as in Section~\ref{4.2.1}.
 To construct $Y_l$, we uniformly sample $1$ vertex for each of the $5$ classes, i.e. $m = 5$.

\subsubsection{Citation networks} We also evaluate the Cora citation network, demonstrating that our method generalizes beyond $k$-NN graphs. In this benchmark, the graph is given. The vertices represent documents and their links refer to citations between documents, while the label corresponds to the topic of the document.

\subsection{Numerical Results}

We compare SSM with R-Trust Region, R-Conjugate Gradient, R-Gradient, Landing, and RSDM on the reduced Stiefel problem obtained from Proposition~\ref{prop 1.1} and Remark~\ref{last}. 
For a fair comparison,  all methods start with the same initialization $\cP_\cM (V_g V_g^\top B)$ to avoid crashing at potential saddles. \footnote{Random initialization produces poor objective values in the methods:  R-TRUST REGION, R-CONJUGATE GRADIENT, R-GRADIENT, RSDM, and LANDING. We omit these results. }
We report the final objective value, first-order residual 
\[
    \|AXC-B-X\Lambda_{\rm sym}\|,
\]
the number of \(L^\dagger x\) calls, wall-clock runtime, and classification accuracy. The behavior of the objective value and the first-order residual of these methods is reported in Fig.~\ref{fig:convergence_time} and Fig.~\ref{fig:grad_time}, respectively. 

\begin{figure*}[t!]
    \centering
    \includegraphics[width=0.57\textwidth]{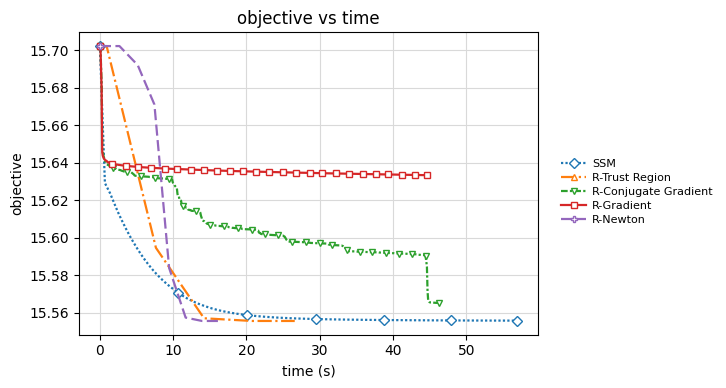}
    \includegraphics[width=0.42\textwidth]{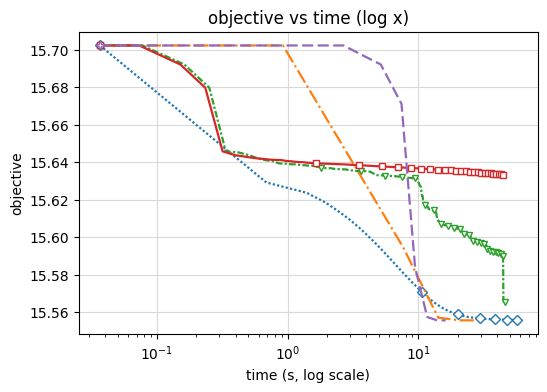}
    \caption{ Objective value versus wall-clock time on the MNIST semi-supervised graph problem.  All methods use the same initialization \(X_1=\mathcal P_{\mathcal M}(V_gV_g^\top B)\).  \textbf{Left:} linear time scale.  \textbf{Right:} logarithmic time scale, which emphasizes early-time behavior.  SSM, R-Trust Region, and R-Newton reach the lowest objective range in the plotted time window, whereas R-Conjugate Gradient and R-Gradient remain at higher objective values. }
    \label{fig:convergence_time}
\end{figure*}

\begin{figure*}[t!]
    \centering
    \includegraphics[width=0.57\textwidth]{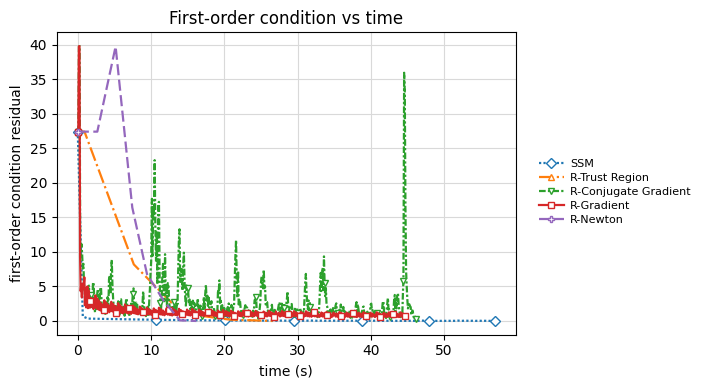}
    \includegraphics[width=0.42\textwidth]{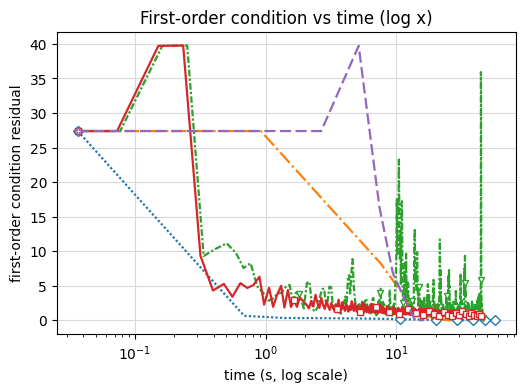}
    \caption{ First-order residual versus wall-clock time on the MNIST semi-supervised graph problem.  \textbf{Left:} linear time scale.  \textbf{Right:} logarithmic time scale, which emphasizes rapid, early-time convergence behavior of SSM.  }
    \label{fig:grad_time}
\end{figure*}
{Figure~\ref{fig:convergence_time} compares wall-clock convergence on the MNIST semi-supervised graph problem. 
 SSM reduces the objective immediately from the common initialization and reaches the lowest objective range attained in the experiment. 
R-Trust Region and R-Newton also attain comparable final objective values, but their curves show a different time profile, with a short initial plateau followed by a sharp decrease. 
R-Conjugate Gradient decreases more gradually and remains above the best final objective values, while R-Gradient stagnates at a substantially larger objective. 
The logarithmic time plot highlights these early-time differences.}

Table ~\ref{tab:imagedatasets} compares SSM with the baseline methods, including R-TRUST REGION, R-CONJUGATE GRADIENT, R-GRADIENT, RSDM, and LANDING
  to solve $X_u$  from (\ref{eq_199}). \footnote{ The code and experiments can be found at
https://github.com/choltz95/SSM-on-stiefel-manifolds.}
\begin{itemize}
\item SSM is Algorithm~\ref{SSM_alg} in section~\ref{sec:3}.
\item R-TRUST REGION is the trust region method described in section 6.4\cite{boumal_introduction_2022}.  
\item  R-CONJUGATE GRADIENT is the Riemannian conjugate gradient method proposed in\cite{Sato2022}.
\item  R-GRADIENT is the iterative method given  in (\ref{eq_97}). 
\item {LANDING is the infeasible landing method for optimization under orthogonality constraints; it uses cheaper iterations that are attracted toward the Stiefel manifold rather than enforcing feasibility exactly at every step~\cite{ablin2022Landing}.}
\item {RSDM is the Randomized Riemannian Submanifold Method of Han, Poirion, and Takeda~\cite{HanPoirionTakeda2025RSDM}; it restricts each Riemannian update to a randomly selected low-dimensional submanifold.}
\end{itemize}

The computation of $V_g$ is necessary to compute both the initialization of SSM and is used to compute the sequence of subspaces. We employ  locally optimal block preconditioned conjugate gradient methods (LOBPCG)\cite{Knyazev2001}, a matrix-free algorithm with Jacobi preconditioning to compute the principal eigenvectors of $A$, $V_g$ in nearly linear time.  
For instance,   
we compute the smallest  four eigenvalues and corresponding eigenvectors of $A$   to generate $V_g$ with rank $4$  for the MNIST dataset.
Here, the first $6$ positive eigenvalues of $A$ for the MNIST dataset are
\begin{eqnarray}
&& d_1 = 4.9439\times 10^{-4},\:\: d_2 = 9.7915 \times 10^{-4},\:\:   d_3 = 1.1816\times 10^{-3},\\
&& d_4 =  1.2345\times 10^{-3},\:\: d_5 = 2.1650\times 10^{-3},\:\: d_6 = 2.4785\times 10^{-3}.
\end{eqnarray}

\begin{table*}[h!]
\caption{Numerical experiments for image datasets (MNIST, FashionMNIST, and CIFAR-10) and a citation network (CORA) at 1 label rate.
MNIST: $m=r=5, n=35683$.
FashionMNIST: $m=r=5, n=35000$.
CIFAR-10:  $m=r=5, n=30000$.
 CORA :  $m=r=7, n=2708$ 
}
\label{tab:imagedatasets}
\begin{center}
\begin{small}
\begin{sc}
\begin{tabular}{llllll}
\toprule
MNIST  &$\mbf(X)$&$|| AXC - B-X\Lambda_{sym}||$&\# evaluations $L^{\dagger} x$ &runtime (s) & accuracy \\
\midrule
SSM & 15.57 & 0.0 & 19 &  33.1 & 97.61\%   \\
R-Trust Region & 15.57 & 0.0 & 43  &  97.3  &  97.61\%   \\
R-Conjugate Gradient & 25.08 & 0.07 & 0 &  72.3 & 75.31\%    \\
R-Gradient & 25.08 & 0.09 & 0 &  117.9  & 75.22\%   \\
Landing & 24.49 &1.13 & 0 & 27.9 & 75.34\%\\
RSDM & 28.41 & 12.4 & 0 & 93.2 & 74.18\%\\
\bottomrule
\toprule
FashionMNIST  & & & & &\\
\midrule
SSM & 18.61 & 0.0 & 13 &  74.3  & 53.57\%   \\
R-Trust Region & 18.61 & 0.0 & 21 &  231.4  & 53.57\%   \\
R-Conjugate Gradient & 29.81 & 0.03 & 0 &  93.4 &  21.53\%   \\
R-Gradient & 29.68 & 0.02 & 0 &  87.0 &  29.97\%   \\
Landing & 29.15 &1.25 & 0 & 63.0 & 30.1\%\\
RSDM & 33.75 & 13.5 & 0 & 120.5 & 20.5\%\\
\bottomrule
\toprule
CIFAR-10   &  & & & & \\
\midrule
SSM & 52.76 & 0.0 & 41 & 136.4 & 52.76\%     \\
R-Trust Region & 52.76 & 0.0 & 64 &  397.4  & 52.76\%   \\
R-Conjugate Gradient & 74.83 & 0.01 & 0 &  177.3 &  39.41\%   \\
R-Gradient & 74.61 & 0.02 & 0 &  162.1 &  33.96\%    \\
Landing & 72.95 &1.95 & 0 & 116.0 & 40.0\%\\
RSDM & 84.80 & 21.5 & 0 & 228.5 & 38.0\%\\
\bottomrule
\toprule
CORA     &  & & & & \\
\midrule
SSM & 104.37 & 0.0 & 23 & 48.3 & 61.31\%     \\
R-Trust Region & 104.37 & 0.0 & 49 &  102.6  & 61.31\%   \\
R-Conjugate Gradient & 213.18 & 0.03 & 0 &  72.1 &  37.11\%   \\
R-Gradient & 213.18 & 0.07 & 0 &  79.10 &  37.03\%    \\
Landing & 207.8 &3.30 & 0 & 41.0 & 37.4\%\\
RSDM & 241.6 & 36.1 & 0 & 92.9 & 35.8\%\\
\bottomrule
\end{tabular}
\end{sc}
\end{small}
\end{center}
\end{table*}



Table ~\ref{tab:imagedatasets} indicates that a better classification accuracy is obtained,  if a lower objective is reached.
Compared to the methods R-CONJUGATE GRADIENT and R-GRADIENT, both SSM and R-TRUST REGION can reach a lower objective value, see the column of $\mbf(X)$.
  On the other hand, even though these methods are guaranteed to monotonically reduce the objective via line search, SSM rapidly converges to a qualified critical point, whereas the R-GRADIENT method fails to converge even after hundreds of iterations. 
See the column $\|AXC-B-X\Lambda_{sym}\|$ for the norm of the first-order condition.  Regarding R-CG method, possibly due to nearly singular or indefinite Hessian structure, we observe spikes in the first-order residual, even though the line search mechanism ensures that the objective function continues to decrease.

One may also ask how our method compares to traditional second-order methods (i.e., Riemannian Trust-region). 
The runtime column reports the wall-clock cost. In the reported runs, SSM is faster than R-Trust Region while attaining comparable or lower objective values.   The $L^\dagger x$ column reports the number of calls to the conjugate gradient solver for each method (to compute the Newton direction). 
In theory, SSM employs a special set of vectors to estimate the Hessian information to update the search direction via subspace minimization. As a result, the  Hessian information estimated from SSM is usually better than the Hessian estimated from CG or BFGS methods typically used for trust-region type approaches. Take  MNIST as one example. The smallest eigenvalue of Hessian in SSM is about $1.3\times 10^{-3}$ near the critical point, while the smallest eigenvalue of Hessian in R-TRUST REGION is less than $2.05 \times 10^{-5}$.
%
%


\subsection{Complexity of SSM and outlook}
\label{sec:complexity}

Finally,  we briefly analyze the computational cost of SSM, which is dominated by the SQP routine 
to compute the SQP directions $Z$ (Newton directions). 
The SQP direction $Z$ is the solution to the system characterized by the linearization of the first-order optimality conditions. Namely, within each iteration of our procedure, we compute the Lagrangian multipliers as well as the SQP update for $X$. As in Newton's method for unconstrained problems, SQP-based methods necessitate the computation of inverse-vector products involving symmetric PSD linear systems. 
Assume that by exploiting the sparsity of the graph Laplacian  $L$, vector-vector and matrix-matrix multiplication can be done in linear time. The primary overhead of our method lies in the computation $Z$, which necessitates the calculation of the solution to a specific linear system in the subspace orthogonal to $\text{span}\{ V_g\}$. We empirically show that this linear system is well-conditioned.  In summary, SSM is an efficient tool for computing a qualified critical point of large-dimensional problems in (\ref{main_P}).

 Finding global minimizers in nonlinear optimization often presents significant challenges. In this study, we introduce a novel convex relaxation approach for quadratic minimization on Stiefel manifolds, which significantly improves the quality of solutions. It  raises an interesting question: can similar  results/algorithms be achieved  in a broader context involving nonsmooth objectives? Specifically, we wonder if proximal algorithms\cite{Beck}  can be adapted for use  on Stiefel manifolds. This topic will be explored in future research.

\subsection{Statements and Declarations}

All authors declare no conflicts of interest.

\appendix
\section{Auxiliary propositions}
\subsection{Proof of Necessary conditions for local minimizers}
Here is  the proof of Prop.~\ref{suf_cond}.
Observe that $St(n,r)$ is one sub-manifold of $\cO(n)$. 
According to  (\ref{T_X}),
 we can  introduce a  differentiable curve $\rho(t)$ 
 \beqq\label{eq36}
\rho(t)=[X, X_\bot]\exp(t \Omega) I_{n,r},\; \Omega:=\left(
\begin{array}{cc}
\Delta_0, & -\Delta_1^\top\\
 \Delta_1,&  0
\end{array}
\right)
\eeqq
passing $X$ on $St(n,r)$,
where  $\Delta_1\in \IR^{r\times (n-r)}$ is
a nonzero matrix and $\Delta_0\in \IR^{r\times r}$ is a
 skew-symmetric matrix. 
Indeed,  $\rho(t)^\top \rho(t)=I_{r}$ holds, i.e.,  $\rho(t)\in St(n,r)$ and $\rho(0)=X$, $\rho'(0)=X\Delta_0+X_\bot \Delta_1$.
The following geometric viewpoint indicates that eigenvalues of $\Lambda C^{-1}$ should 
be bounded above by $d_r$.



%

 From (\ref{eq36}), computation shows 
\begin{eqnarray}\label{eq32'}
&&\frac{d}{dt}\mbf(X(t))=\langle \rho'(t), A\rho(t) C\rangle- \langle \rho'(t), B\rangle\\
&&\frac{d^2}{dt^2}\mbf(X(t))=\langle \rho'(t), A\rho'(t) C\rangle+ \langle \rho''(t), A\rho(t)C-B\rangle,
\end{eqnarray}
and
\begin{eqnarray}
&&\rho'(0)=[X, X_\bot] \Omega I_{n,r}= X\Delta_0 +X_\bot\Delta_1,\\
&& \rho''(0)=[X, X_\bot] \Omega^2 I_{n,r}=X(\Delta_0^2 -\Delta_1^\top \Delta_1)+X_\bot \Delta_1 \Delta_0.
\end{eqnarray}
Then  $\mbf(\rho(t))$ is a strictly local minimizer  at $t=0$,    if and only if 
the following two optimal conditions hold.
First,   (\ref{eq32'}) indicates
 (\ref{1stC}), since
\beqq 
\frac{d}{dt}\mbf(\rho(t))|_{t=0}=\langle AXC-B, X\Delta_0 +
X_\bot \Delta_1 
\rangle=0\eeqq
holds for any matrix $\Delta_1 $ and any skew-symmetric matrix  $\Delta_0$. 
Second, let 
 $Y:=X\Delta_0 +X_\bot\Delta_1$. We should have
\begin{eqnarray}
&&\frac{d^2}{dt^2}\mbf(\rho(t))|_{t=0}=\langle \rho''(0), A\rho(0)C-B\rangle+ \langle \rho'(0), A\rho'(0) C\rangle\\
&=&-\langle  Y^\top Y ,\Lambda\rangle+
\langle Y , A Y  C\rangle\ge 0,\label{eq35}
\end{eqnarray}
where we use the first order condition in (\ref{1stC}).
Now   set $\Delta_0=0$ and $\Delta_1=u_1 u_0^\top C^{-1/2} $, where $u_1$ is a unit 
 eigenvector of  $X_\bot^\top A X_\bot$ associated with eigenvalue $d^\bot_{min}$ and $u_0$ is a unit eigenvector of $C^{-1/2} \Lambda C^{-1/2}$.
 Then (\ref{eq35}) implies  (\ref{eq_36}).


\subsection{Maximizers of a quadratic maximization  on Stiefel manifold}

\begin{prop}\label{max_P} Consider the problem
 \beqq\label{problemX}
 \max_{X\in \IR^{r\times r'}} \mbf(X; A, B, C),\; X^\top X \preceq I_{r'},
 \eeqq
where $C\in \IR^{r'\times r'}$ is  positive definite with $r\ge r'$.
\begin{itemize}
\item Suppose 
 $A\in \IR^{r\times r}$ is  positive definite.
 Then any global maximizer  $\hat X$ must lie on $St(r,r')$.
 \item 
 Suppose $A$ is  positive semidefinite. Then there exists 
  $\hat X'\in St(r,r')$, that is a global maximizer. 
\end{itemize}
\end{prop}
\begin{proof} Let $\cD$ denote the set $\{ X: X^\top X \preceq I_{r'}\}$. 
Let  $\{ \gamma_1, \ldots, \gamma_{r'}\}$ be singular values of  the  global maximizer. Then   $\gamma_j\le 1$ for all $j$. 
Without loss of generality, let $\hat X\notin St(r,r')$ be a global maximizer  with
 $\gamma_1<1$.
Take the svd of $\hat X$,
 \beqq
\hat X=[  x_1, \ldots, x_{r'}] \diag(\gamma_1, \ldots, \gamma_{r'}) Q^\top,
\; \textrm{ with } [  x_1, \ldots, x_{r'}] \in St(r,r'),\; Q\in \cO(r').
\eeqq
Construct a path $\{ \mbg(t)\in \cD: \; t\in \IR\}$ through $\hat X$, where
 \beqq
\mbg(t)=[ x_1, \ldots, x_{r'}]\diag( [t, \gamma_2, \ldots, \gamma_{r'}])Q^\top.
\eeqq The constraint
$X^\top X\preceq I$ implies $t\in [-1,1]$.  Since
 $x_1^\top A x_1>0$ and $(Q^\top CQ)_{1,1}>0$,
the maximization of  $\mbf$ in (\ref{problemX}) 
on the interval reduces to  the  concave maximization of $t$.
The optimal value of $t$   is either  $1$ or $-1$.
Hence,  $\hat X$ is not a global maximizer.  This contradiction implies that   each maximizer  should be located on $St(r,r')$.

Suppose $A$ is positive semi-definite and there exists some global maximizer $\hat X\notin St(r,r')$. For each $\gamma_j<1$, 
repeat the  above process of replacing  $\hat X$ with $\mbg(1)$ or $\mbg(-1)$ to obtain the
 global maximum point  located on $St(r,r')$.
\end{proof}

 \subsection*{Acknowledgment} We thank the anonymous referees for valuable comments and suggestions that led to improvement of the original manuscript.
\bibliographystyle{alpha}
\bibliography{placement}

%
%
%
%
%
\end{document}